\documentclass[10pt]{article}
\RequirePackage[left=28mm,right=28mm,top=30mm,bottom=30mm]{geometry}
\usepackage[utf8]{inputenc}
\usepackage[english]{babel}
\usepackage{amsmath}
\usepackage{amssymb}
\usepackage{amsthm}
\usepackage[shortlabels]{enumitem} 
\usepackage{natbib}
\usepackage{afterpage}
\usepackage{graphicx} 
\usepackage{tikz-cd} 
\usepackage{stmaryrd} 
\usepackage{url}
\usepackage{hyperref}
\usepackage{multicol}

\usetikzlibrary{shapes.geometric,arrows.meta}   

\definecolor{projcol}{RGB}{0,85,160}    
\definecolor{tiltcol}{RGB}{200,30,60}   
\definecolor{injcol}{RGB}{0,125,70}     

\tikzset{
	stdmark/.style   ={regular polygon, regular polygon sides=3, draw=black!75,
		line width=.6pt, inner sep=0pt},
	costdmark/.style ={regular polygon, regular polygon sides=3, shape border rotate=180,
		draw=black!75, line width=.6pt, inner sep=0pt},
	mod/.style       ={inner sep=2pt, font=\normalsize},
	ar/.style        ={-{Stealth[length=2mm]}, shorten >=2pt, shorten <=2pt},
	taua/.style      ={dashed, -{Stealth[length=2mm]}, gray, shorten >=2pt, shorten <=2pt},
}

\newcommand{\End}{\operatorname{End}}
\newcommand{\Hom}{\operatorname{Hom}}
\newcommand{\gldim}{\operatorname{gldim}}
\newcommand{\injdim}{\operatorname{injdim}}
\newcommand{\Ext}{\operatorname{Ext}}
\newcommand{\Tor}{\operatorname{Tor}}
\newcommand{\add}{\!\operatorname{add}}

\newcommand{\m}{\!\operatorname{-mod}} 

\newcommand{\proj}{\!\operatorname{-proj}}
\newcommand{\inj}{\!\operatorname{-inj}}

\newcommand{\id}{\!\operatorname{id}}
\newcommand{\St}{\Delta}
\newcommand{\gen}{\operatorname{Gen}}
\newcommand{\rad}{\operatorname{rad}}
\newcommand{\cogen}{\operatorname{Cogen}}
\newcommand{\Cs}{\nabla}
\renewcommand{\L}{\Lambda}
\renewcommand{\l}{\lambda}
\newcommand{\Tr}{\operatorname{Tr}}
\renewcommand{\Im}{\operatorname{Im}}
\renewcommand{\top}{\operatorname{top}}

\newcommand{\dv}{\underline{\dim}\,}          
\newcommand{\Loe}[1]{\begin{matrix}#1\end{matrix}}   

\newtheorem{numberingthm}{Theorem}[section] 
\theoremstyle{definition}
\newtheorem{Def}[numberingthm]{Definition}
\newtheorem{example}[numberingthm]{Example}

\theoremstyle{plain}
\newtheorem{Prop}[numberingthm]{Proposition}
\newtheorem{Theorem}[numberingthm]{Theorem}

\newtheorem{Cor}[numberingthm]{Corollary}
\newtheorem{Lemma}[numberingthm]{Lemma}

\newtheorem{Remark}[numberingthm]{Remark}
\newtheorem{claim}{Claim}
\newtheorem{thmintroduction}{Theorem}
\newtheorem{corintroduction}{Corollary}

\newenvironment{claimproof}[1][Proof of claim]
{\begin{proof}[#1]}
	{\end{proof}}

\newenvironment{Example}
{\pushQED{\qed}\example}
{\popQED\endexample}

\providecommand{\keywords}[1]
{\scriptsize
	\textbf{\textit{Keywords:}} #1 \normalsize \hfill
}
\providecommand{\msc}[1]
{\scriptsize
	\textbf{\textit{2020 Mathematics Subject Classification:}} #1 \normalsize \hfill
}

\title{\vspace*{-1cm}Ringel self-duality of hereditary algebras \\via Auslander-Reiten theory} 
\author{Tiago Cruz} 
\date{}

\newcommand{\Address}{{
		\bigskip
		\footnotesize
		
		TIAGO CRUZ,\par \textsc{Institute of Algebra and Number Theory}\par \textsc{University of Stuttgart,}\par \textsc{Pfaffenwaldring 57, 70569 Stuttgart, Germany,}\par\nopagebreak
		\textit{E-mail address}, T.~Cruz: \texttt{tiago.cruz@mathematik.uni-stuttgart.de}
		}}
		
		\allowdisplaybreaks

\begin{document}

\maketitle

\begin{abstract}Highest weight categories with finitely many simples correspond to quasi-hereditary algebras. Ringel duality shows that, up to Morita equivalence, quasi-hereditary algebras come in pairs, while Ringel self-duality is the phenomenon in which a quasi-hereditary algebra is paired with itself.
The main purpose of this paper is to show that there are connected non-semisimple hereditary algebras that are Ringel self-dual and to give a combinatorial interpretation of Ringel self-duality in terms of Auslander-Reiten theory. 

Let $A$ be a basic connected hereditary algebra over an algebraically closed field $k$ equipped with a quasi-hereditary structure. We show that, if $A$ is of finite representation type and Ringel self-dual, then the indecomposable summands of the characteristic tilting module $T$ are obtained from the indecomposable projectives, and from indecomposable injectives by applying powers of the Auslander-Reiten translation, in a way governed by an automorphism of the underlying Dynkin diagram. In particular, the number of indecomposable $A$-modules has the same parity as the number of simple $A$-modules. Conversely, we show that if $T\cong \tau^{-n}(A)\cong \tau^n\Hom_k(A, k)$ for some natural number $n$, then $A$ is Ringel self-dual and of finite representation type.

Using these results, we determine the simply laced Dynkin diagrams admitting an orientation and a quasi-hereditary structure for which the corresponding path algebra is Ringel self-dual: they are exactly those of type $D_{2m}$ with $m\geq 2$ and $A_{4t+1}$ with $t\geq 0$. In particular, no exceptional type occurs, 
and no connected non-semisimple hereditary  Nakayama   algebra is Ringel self-dual. 
\end{abstract}

\section{Introduction}\enlargethispage{\baselineskip}

Quivers and their representations appear in many areas of mathematics and path algebras of finite acyclic quivers are the standard examples of hereditary algebras. A finite-dimensional algebra is called \emph{hereditary} if its global dimension is at most one or equivalently if every submodule of a projective module is projective. Over algebraically closed fields, hereditary connected algebras are Morita equivalent to the path algebras of finite, connected, acyclic quivers (see for example \citep[VII, Theorem 1.7]{zbMATH02228448}). 

In modern representation theory, hereditary  finite-dimensional algebras played an important role in its development and they are among the best-understood classes of finite-dimensional algebras. Moreover, their representation theory is closely connected with root systems and Dynkin diagrams through Gabriel's theorem.
Gabriel's theorem \cite{zbMATH03366955} states that a path algebra of a finite acyclic quiver has finitely many indecomposable modules, up to isomorphism, if and only if its underlying graph is a simply laced Dynkin diagram (that is of type $A$, $D$ or $E$). Soon after Gabriel's theorem, in \cite{zbMATH03438937}, Bernstein, Gelfand and Ponomarev introduced reflection functors, providing a new proof of Gabriel's theorem relating path algebras arising from different orientations of the same underlying Dynkin graph. On dimension vectors, these functors correspond to the simple reflections generating the Weyl group of the underlying graph. The reflection functors were the starting point of tilting theory. Tilting theory in turn brought derived categories into representation theory (see for instance \cite{zbMATH00040982}). Hereditary algebras and Dynkin quivers also provided developments in other directions, namely connecting representation theory with Lie theory and algebraic combinatorics. For instance, Ringel showed in \cite{zbMATH00019935} that Hall algebras of Dynkin quivers realise positive parts of quantised enveloping algebras while cluster categories built from hereditary algebras categorify cluster algebras \cite{zbMATH05047155}. 

Motivated by highest weight categories arising in Lie theory, Cline Parshall and Scott introduced in \cite{MR961165} the concept of a quasi-hereditary algebra.
In \cite{Dlab1989d}, Dlab and Ringel proved that hereditary algebras admit a quasi-hereditary structure and furthermore that these algebras are exactly the ones which are quasi-hereditary for every choice of total order on the set of non-isomorphic simple modules. 
Roughly speaking, a finite-dimensional algebra is quasi-hereditary if it  admits a quasi-hereditary structure, encoded by a suitable partial order on the set of isomorphism classes of simple $A$-modules, or equivalently by the resulting collection of standard modules. Two such structures on $A$ are regarded as equivalent when their standard modules coincide, equivalently when they determine the same category of modules filtered by standard modules. 
A fundamental feature of the theory of quasi-hereditary algebras is Ringel duality (see for instance \cite{MR1128706, zbMATH06894495, zbMATH07073901, zbMATH07807561} and the references therein.) Associated with every quasi-hereditary algebra $A$ there is another quasi-hereditary algebra, its Ringel dual $R$, defined up to Morita equivalence as the endomorphism algebra of the (basic) characteristic tilting module and the Ringel dual of $R$ is Morita equivalent to $A$.  In particular,  $A$ and its Ringel dual have equivalent derived categories.
  In algebraic Lie theory, the direct summands of characteristic tilting modules have been studied intensively both as tools and objects of interest in their own right. For example, Donkin studied them for reductive groups and Schur algebras \cite{zbMATH00549737}, Andersen for quantum groups at roots of unity \cite{zbMATH00095698} and Riche and Williamson expressed in \cite{zbMATH06898308} their character formulas for $\mathrm{GL}_n$ in terms of the $p$-canonical basis.

While Ringel duality always produces a derived equivalent algebra, Ringel self-duality is a stronger phenomenon in which the Ringel dual is already Morita equivalent to the original algebra.
	 Indeed, we say that $A$ is \emph{Ringel self-dual} if $A$ and its Ringel dual are Morita equivalent in such a way that their quasi-hereditary structures are equivalent (see Subsection \ref{The characteristic tilting module and Ringel duality} below). 
 There are many prominent examples of Ringel self-duality especially in algebraic Lie theory. Examples include many families of Schur and $q$-Schur algebras \cite{zbMATH00549737, zbMATH01235890} \cite{zbMATH01719244} (see also \cite{zbMATH07679240}), the blocks of the BGG category $\mathcal{O}$ of a semi-simple complex Lie algebra \citep[Corollary 2.3]{zbMATH01267662}, \cite{zbMATH01542484} (see also \cite{zbMATH07957365}), quasi-hereditary algebras related to Schubert polynomials \cite{zbMATH06636979}, extended zigzag Schur algebras \cite{zbMATH07839654}, Auslander algebras of self-injective Nakayama algebras \cite{zbMATH06369685} and generalisations (see also \cite{zbMATH07073901, zbMATH06893122}).

Surprisingly it appears  that the problem of Ringel self-duality for hereditary algebras has not been addressed in the literature so far. Since   the Ringel dual of a hereditary algebra is generally a tilted algebra (not necessarily hereditary), this raises the following question:
\begin{enumerate}[(Q)]
	\item Does Ringel self-duality occur for (non-semisimple) hereditary algebras? If so, what type of symmetry does Ringel self-duality impose on the Gabriel quiver and the Auslander-Reiten quiver?
\end{enumerate}

The aim of this paper is to show that Ringel self-duality does occur for connected hereditary algebras and to give a combinatorial interpretation of this phenomenon in terms of the Auslander-Reiten quiver. In finite representation type, Ringel self-duality imposes a symmetry on the positions of the direct summands of the characteristic tilting module, governed by an automorphism of the underlying Dynkin diagram. At the same time, it partitions the indecomposable modules into three disjoint classes: those generated by $\tau^{-1}T$, the indecomposable summands of $T$, and those cogenerated by $\tau T$, where $T$ is the characteristic tilting module and $\tau$ is the Auslander-Reiten translation.  Our main theorem can be stated as follows: \enlargethispage{1.2\baselineskip}

\begin{thmintroduction}\label{maintheorem} (see Lemma \ref{QuiveraddT}, Proposition \ref{counting} and Theorem \ref{maintheoremoneimplication})
Let $(A, \Lambda, \leq)$ be a (connected and basic) quasi-hereditary $k$-algebra with $\gldim A\leq 1$ and $k$ an algebraically closed field. Let $T$ be the (basic) characteristic tilting module and let $P(i)$ (resp. $I(i)$) be the projective (resp. injective) indecomposable module associated with $i\in \L$. The following assertions hold. 
\begin{enumerate}[(1)]
	\item If $A$ is an algebra of finite representation type and $(A, \Lambda, \leq)$ is Ringel self-dual, then there exists a graph automorphism $\sigma$ of the underlying graph of the Gabriel quiver of $A$ and non-negative integers $p_i$, for $i\in \L$, such that  $$T\cong\bigoplus_{i\in \Lambda} \tau^{-p_i}P(i)\cong \bigoplus_{i\in \L} \tau^{p_{\sigma^{-1}(i)}}I(i).$$
	Moreover the number of non-isomorphic indecomposable $A$-modules has the same parity as the cardinality of $|\L|$.
	\item If $T\cong \tau^{-n}(A)\cong\tau^n(\Hom_k(A, k))$, then $(A, \Lambda, \leq)$ is Ringel self-dual and it is an algebra of finite representation type.
\end{enumerate}

\end{thmintroduction} 

The proof of our main result combines tilting theory with Auslander-Reiten theory. Using Brenner-Butler theory, we show that Ringel self-duality forces the two torsion pairs associated with the characteristic tilting module to split, yielding a description of the Auslander-Reiten quiver of the algebra in terms of category of modules filtered by standard modules and the category of modules filtered by costandard modules. Further, in this setting, the characteristic tilting module $T$ is isomorphic to $\tau^{-n}(A)$ precisely when the category of modules filtered by standard modules is the additive closure of the modules $T, \tau (T), \ldots, \tau^n (T)$.

A practical advantage of Theorem \ref{maintheorem} is that it allows us to test Ringel self-duality on a hereditary algebra without explicitly computing the endomorphism algebra of the characteristic tilting module. 
As application of Theorem \ref{maintheorem}, we get the following.

\begin{thmintroduction} (see Theorem \ref{maintheoremBproof}) Let $\Gamma$ be a simply laced Dynkin diagram. \label{maintheoremB}	
	 There exists a quiver $Q$ with underlying graph $\Gamma$ and a partial order $\leq$ on $Q_0$ making the corresponding quasi-hereditary algebra $(kQ, Q_0, \leq)$ Ringel self-dual if and only if $\Gamma$ is of type $D_{2m}$ or $A_{4t+1}$ for integers $m\geq 2$ and $t\geq 0$.
\end{thmintroduction}

So Ringel self-duality never occurs in type $E$ while in type $D$, a parity phenomenon occurs: Ringel self-duality is possible for Dynkin diagrams of even type $D_{2m}$, but impossible for those of odd type. 

Many Ringel self-dual algebras known in the literature (for instance blocks of the BGG category $\mathcal{O}$) admit a simple preserving duality and have positive dominant dimension. By contrast, Theorem \ref{maintheoremB} provides explicit families of Ringel self-dual algebras that admit no simple preserving duality and have dominant dimension zero (since the linear orientation is not allowed in type $A$). To the best of our knowledge, Theorem \ref{maintheoremB} provides the first examples of Ringel self-dual algebras without a simple preserving duality to be explicitly identified in the literature.

Another consequence of the combination of Theorems \ref{maintheorem} and \ref{maintheoremB} is that Ringel self-duality forces all $\tau$-orbits of a representation finite hereditary algebra to have the same cardinality. As a consequence, outside type $D_{2m}$ with $m>2$,  Theorem \ref{maintheorem}  admits the following sharper formulation.

\begin{corintroduction}(see Corollary \ref{corCinternal}) \label{corC}
	Let $(A, \Lambda, \leq)$ be a (connected and basic) quasi-hereditary $k$-algebra with $\gldim A\leq 1$ whose Gabriel quiver is not of type $D_{2m}$ for $m>2$ and $k$ is an algebraically closed field. Then
	$(A, \Lambda, \leq)$ is Ringel self-dual and it is an algebra of finite representation type if and only if the (basic) characteristic tilting module $T$ is isomorphic to $\tau^{-n}(A)\cong\tau^n(\Hom_k(A, k))$ for some $n\in \mathbb{N}\cup \{0\}$.
\end{corintroduction}

In Proposition \ref{proptwistedcaseD}, we establish a one-sided analogue of Corollary  \ref{corC} for type $D_{2m}$, $m>2$, where the corresponding description of the characteristic tilting module is  modified by a twist arising from the non-trivial graph automorphism of the Dynkin diagram.

The article is organised as follows: In Section \ref{Preliminaries} we introduce the notation, terminology and standard facts to be used throughout the paper on Auslander-Reiten theory (see Subsection \ref{Auslander-Reiten quiver of a category}), on hereditary algebras (see Subsection \ref{Hereditary algebras over algebraically closed fields}), on quasi-hereditary algebras (see Subsection \ref{Quasi-hereditary structures}), and on torsion pairs (see Subsection \ref{Torsion pairs}). In Subsection \ref{Two split torsion pairs associated with the characteristic tilting module}, we establish that the two torsion pairs asssociated to the characteristic tilting module are split when the algebra is Ringel self-dual. In particular, that every indecomposable over a Ringel self-dual algebra has a filtration by standard modules or a filtration by costandard modules. In Subsection \ref{Proof of the main result}, we give the proof of Theorem \ref{maintheorem}. In Section \ref{Applications and classification}, we apply the methods developed in Section \ref{Properties} to the path algebras of type $A$, $D$ and $E$. Namely, in Section \ref{Applications and classification}, we get Theorem \ref{maintheoremB} and Corollary \ref{corC} as application of Theorem \ref{maintheorem}. In particular, in Subsection \ref{Examples of hereditary algebras that are not Ringel self-dual}, we obtain that no non-semisimple connected Nakayama hereditary algebra is Ringel self-dual.

\section{Preliminaries} \label{Preliminaries}

We will assume throughout this paper that $k$ is an algebraically closed field and $A$ and $B$ are basic finite-dimensional $k$-algebras.
By $A\m$  we mean the category of finitely generated  left  $A$-modules and by $A\proj$ the full subcategory of $A\m$ whose modules are the finitely generated projective $A$-modules. Given $M\in A\m$, we denote by $\operatorname{soc} M$ the socle of $M$ whereas we denote by $\top M$ the top of $M$. We denote by $\add_A M$ (or just $\add M$ when $A$ is fixed) the full subcategory of $A\m$ whose modules are direct summands of finite direct sums of $M\in A\m$.  We write $A\proj$ to denote $\add A$. Given $M\in A\m$ we write $|M|$ to denote the number of non-isomorphic indecomposable $A$-modules of $M$. For a set $X$, we write $|X|$ to denote the cardinality of $X$. The opposite algebra of $A$ will be denoted by $A^{op}$.  By $D$ we denote the standard duality $\Hom_{k}(-, k)\colon A\m\longleftrightarrow A^{op}\m$.
Given a full subcategory $\mathcal{X}$ of $A\m$ we denote by $D\mathcal{X}$ the full subcategory of $A^{op}\m$ whose objects are of the form $DX$ for some $X\in \mathcal{X}$. Given a full subcategory $\mathcal{X}$ of $A\m$, we write $\mathcal{X}^\perp$ to denote the category $\{M\in A\m\colon \Ext_A^i(X, M)=0, \ \forall X\in \mathcal{X}, \ \forall i>0\}$ and $\mathcal{X}^{\perp_1}$ to denote the category $\{M\in A\m\colon \Ext_A^1(X, M)=0, \ \forall X\in \mathcal{X}\}$. Similarly, one defines ${}^\perp\mathcal{X}$ and ${}^{\perp_1}\mathcal{X}$.
The global dimension of $A$ is denoted by $\gldim A$.
For a quiver $Q$, we denote by $Q_0$ its set of vertices and $Q_1$ the set of arrows. 
Throughout this paper, we use the convention that paths in a quiver are composed as morphisms, that is, from right to left.

\subsection{Auslander-Reiten quiver of a category}\label{Auslander-Reiten quiver of a category}

Let $\mathcal{A}$ be an additive full subcategory of $A\m$ that is closed under direct sums, summands and extensions. We can regard $\mathcal{A}$ as an exact category by considering the exact sequences inherited from $A\m$ as the conflations.

A morphism $f\colon X\rightarrow Y$ in $\mathcal{A}$ is a \emph{split monomorphism} (resp. \emph{split epimorphism}) if there exists $r\in \Hom_{\mathcal{A}}(Y, X)$ such that $r\circ f=\id_{X}$ (resp. $f\circ r=\id_Y$). The morphism $f\in \Hom_{\mathcal{A}}(X, Y)$ is called \emph{irreducible} if $f$ is neither a split monomorphism nor a split epimorphism, and if $f=h\circ g$ for some $h\in \Hom_{\mathcal{A}}(Z, Y)$ and $g\in \Hom_{\mathcal{A}}(X, Z)$, then $g$ is a split monomorphism or $h$ is a split epimorphism. Given $X, Y\in \mathcal{A}$, the irreducible morphisms from $X$ to $Y$ give rise to a $k$-vector space, denoted as $\operatorname{Irr}_\mathcal{A}(X, Y)$.
(see for instance  \citep[A.3.]{zbMATH02228448}.)

\begin{Lemma}\label{irreduciblefullsubcategories}
	Let $\mathcal{B}$ be an additive full subcategory of $\mathcal{A}$ closed under direct sums and summands. Let $X, Y\in \mathcal{B}$ and $f\in \Hom_{\mathcal{A}}(X, Y)$. Then $f$ is a split monomorphism (resp. split epimorphism) in $\mathcal{A}$ if and only if it is in $\mathcal{B}$. Further, if $f$ is irreducible in $\mathcal{A}$, then it is irreducible in $\mathcal{B}$.
\end{Lemma}
\begin{proof}
	Since $\mathcal{B}$ is full in $\mathcal{A}$, the first statement is clear. Moreover, any factorisation in $\mathcal{B}$ is a factorisation in $\mathcal{A}$, thus the second statement is also clear.
\end{proof}

The \emph{Auslander-Reiten quiver} of $\mathcal{A}$, written as $\Gamma(\mathcal{A})$ is the quiver defined as follows: the vertices of $\Gamma(\mathcal{A})$ are in one to one correspondence with the indecomposable objects of $\mathcal{A}$ and are denoted by $[M]$, while the arrows $[M]\rightarrow [N]$ are in one to one correspondence with the vectors of a basis of $\operatorname{Irr}_\mathcal{A}(X, Y)$.

A morphism $g\colon Y\rightarrow Z$ in $\mathcal{A}$ is \emph{right almost split} if it is not a split epimorphism and for any morphism $h\colon C\rightarrow Z$ which is not a split epimorphism there exists $r_h\colon C\rightarrow Y$ such that $h=g\circ r_h$. Dually, a morphism $f\colon X\rightarrow Y$ in $\mathcal{A}$ is \emph{left almost split} if it is not a split monomorphism and any morphism $X\rightarrow C$ which is not a split monomorphism factors through $f$. 
An exact sequence in $\mathcal{A}$, $0\rightarrow X\xrightarrow{f} Y\xrightarrow{g} Z\rightarrow 0$ is called an \emph{Auslander-Reiten sequence} in $\mathcal{A}$ if $f$ is left almost split and $g$ is right almost split (see for instance  \citep[Proposition 2.3]{zbMATH01224795}).
We say that $\mathcal{A}$ \emph{has Auslander-Reiten sequences} if the following three conditions are satisfied: first, every indecomposable object $X\in \mathcal{A}$ there exists  a right almost split morphism ending at $X$ and a left almost split morphism starting at $X$; secondly for all non-projective indecomposable objects $Z$ in $\mathcal{A}$  there exists an Auslander-Reiten sequence $0\rightarrow X\rightarrow Y\rightarrow Z\rightarrow 0$ in $\mathcal{A}$, and third, if for all indecomposable non-injective objects $X$ there exists an Auslander-Reiten sequence $0\rightarrow X\rightarrow Y\rightarrow Z\rightarrow 0$ in $\mathcal{A}$.

When $\mathcal{A}$ has Auslander-Reiten sequences, every irreducible morphism occurs as a component of a left or right almost split morphism, and, whenever the corresponding Auslander-Reiten sequence exists, as components of its maps (see for instance \citep[V, Theorem 5.3]{zbMATH00707210} for the classical case and \citep[Theorem 1.4]{zbMATH06059246} for the general case). \enlargethispage{\baselineskip}

For a finite-dimensional algebra $A$, we write $\tau=D\Tr$ and $\tau^{-1}=\Tr D$ for the Auslander-Reiten translations in $A\m$. For instance, if $Z$ is indecomposable non-projective $A$-module, the Auslander-Reiten sequence ending at $Z$ has the form $0\rightarrow \tau Z\rightarrow Y\rightarrow Z\rightarrow 0$.
For general background on Auslander-Reiten theory of finite-dimensional algebras we refer to \cite{zbMATH02228448, zbMATH00707210}.

\subsection{Hereditary algebras over algebraically closed fields} \label{Hereditary algebras over algebraically closed fields}
A finite-dimensional algebra $A$ is called \emph{hereditary} if every submodule of a projective $A$-module is projective, equivalently, if $\gldim A\leq 1$.
Hereditary algebras are among the best understood algebras and one particular feature that they possess is that their representation theory is particularly accessible from the viewpoint of Auslander-Reiten theory.
Every basic connected finite-dimensional hereditary algebra (over an algebraically closed field $k$) is isomorphic to the path algebra $kQ$ of a finite acyclic quiver $Q$. Over an arbitrary field, quivers must be replaced by species, but we will not pursue such direction here. By Gabriel's theorem, $kQ$ is of finite representation type (that is $kQ\m=\add M$ for some $M\in kQ\m$) exactly when the underlying graph of $Q$ is a Dynkin diagram of type 
$A_n, D_n, E_6, E_7$ or $E_8$, for $n\geq 1$ (see for instance \citep[VII, Theorem 5.10]{zbMATH02228448}). In particular, the number of isomorphism classes of indecomposable $kQ$-modules are \begin{align}
	\frac{n(n+1)}{2}, \quad n(n-1), \quad 36, \quad 63, \quad 120,
\end{align} respectively.

Let $Q$ be a finite acyclic quiver. Given $M\in kQ\m$, we write $\dv M\in \mathbb{Z}^n$ for the (column) dimension vector of $M$ with respect to the complete set of primitive idempotents $\{e_1, \ldots, e_n\}$ where $e_j$ denotes the primitive idempotent corresponding to $j\in Q_0$. In particular, the vectors $\dv \top kQe_1, \cdots, \dv \top kQ e_n$ form the standard basis of $\mathbb{Z}^n$. 
 If $kQ$ is of finite representation type and $M, N\in kQ\m$ are indecomposable modules, then $\dv M=\dv N$ if and only if $M\cong N$ (see for instance \citep[VII, Theorem 5.10]{zbMATH02228448} or \citep[IX. Lemma 1.1 and Proposition 3.1]{zbMATH02228448}). If $C$ denotes the Cartan matrix of $kQ$, the \emph{Coxeter matrix} of $kQ$ is the invertible matrix $\Phi=-C^TC^{-1}$ so that $\dv \tau M=\Phi \dv M$ for every $M\in kQ\m$ with $\Hom_{kQ}(M, kQ)=0.$ In particular, $\Phi^{-1} \dv M= \dv \tau^{-1} M$ for every $M\in kQ\m$ with $\Hom_{kQ}(D(kQ), M)=0$ (see for instance \citep[IV, Corollary 2.9]{zbMATH02228448}).

Furthermore, when $A=kQ$ is of finite representation type, the indecomposable $A$-modules decompose into finitely many $\tau$-orbits, each of which containing precisely one indecomposable projective and one indecomposable injective module (see for instance \citep[VIII, Example 1.3(a) and VII, Proposition 5.13]{zbMATH02228448}).
Namely, for $i\in Q_0$, the $\tau$-orbit of $Ae_i$ is $\{\tau^{-n}Ae_i\colon n\in \mathbb{N}\cup \{0\}, \ \tau^{-n}Ae_i\neq 0\}$ and its length is the cardinality of this set.
The lengths of these $\tau$-orbits are closely related to the Coxeter number of the underlying Dynkin diagram. More precisely, the orbits can be paired, possibly with themselves, so that the sum of the lengths of each pair is the Coxeter number.

		\begin{Lemma}\label{sizetauorbitsADE}
			Let $Q$ be a Dynkin quiver of type $A_n$, $D_n$, $E_6, E_7$ or $E_8$ and let $A=kQ$. For $i\in Q_0$ let $P(i)$ (resp. $I(i)$) be the projective (resp. injective) indecomposable associated with the vertex $i$ and set $l_i:=|\{n\in \mathbb{N}\cup \{0\}\colon \tau^{-n}P(i)\neq 0 \}|.$
			Then, there is an automorphism $\sigma$ of the underlying Dynkin diagram of $Q$ such that $l_i+l_{\sigma(i)}=h_Q$ and $\tau^{-l_i+1}P(i)\cong I(\sigma(i))$, 
			where $h_Q$ is the Coxeter number of $Q$, namely $n+1$, $2n-2$, $12$, $18$ and $30$ in types $A_n$, $D_n$, $E_6$, $E_7$ and $E_8$, respectively. In particular, $\sigma(i)=n-i+1$ in type $A_n$, while $\sigma$ is the identity in type $D_{2n}$.
		\end{Lemma}
		\begin{proof}
			We refer to \citep[Proposition 4.4.1 and Proposition 4.4.2]{zbMATH08229732}. For the particular case see \citep[Theorem 4.5.1]{zbMATH08229732} or \cite[6.5]{zbMATH03695413}.
		\end{proof}

For our purposes, it is convenient to measure the position of a given indecomposable module in its $\tau$-orbit by its distance to the projective end. When $A$ is a hereditary algebra of finite representation type the value $m(X):=\min \{r\geq 0\colon \tau^r X\in A\proj\}$ is well defined for every indecomposable module $X\in A\m$. The following describes how this value behaves with respect to morphisms and extensions. 
	
	\begin{Lemma}\label{taumeasure}
		Let $A$ be a hereditary algebra of finite representation type. The following assertions hold.
		\begin{enumerate}[(a)]
			\item If $\Hom_A(X, Y)\neq 0$, then $m(X)\leq m(Y)$ for every indecomposable $A$-modules $X$ and $Y$.
			\item If $\Ext_A^1(X, Y)\neq 0$, then $m(Y)<m(X)$ for every indecomposable $A$-modules $X$ and $Y$.
		\end{enumerate} 
	\end{Lemma}
	\begin{proof}
		Since $A$ is representation-directed, $\Hom_A(X, Y)\neq 0$ forces a path of irreducible maps $X\rightarrow \cdots \rightarrow Y$ in $\Gamma(A\m)$. Every arrow in $\Gamma(A\m)$ stays in the column or moves to the next one, so $m$ is non-decreasing along paths. So (a) holds. Assume that $\Ext_A^1(X, Y)\neq 0$. So $X$ is not projective. By the Auslander-Reiten formula, $0\neq \Ext_A^1(X, Y)\cong \Hom_A(Y, \tau X)$. By (a), $m(Y)\leq m(\tau X)=m(X)-1$.
	\end{proof} 

In the hereditary setting the Auslander-Reiten translations can be computed directly without passing through a minimal projective presentation: namely, if  $A$ is a finite-dimensional hereditary $k$-algebra, then 
$\tau=D\Ext_A^1(-, A)$ and $\tau^{-1}=\Ext^1_{A^{op}}(D-, A)$ as functors on $A\m$ (see for instance \citep[VII, Corollary 1.9]{zbMATH02228448}). An easy consequence is that the vanishing of powers of $\tau$ can be detected by checking it only on simple modules.
	
	\begin{Lemma}\label{tauvanishingonsimples}
		Let $A$ be a hereditary algebra. If $\tau^m S=0$ for every simple $A$-module, then $\tau^m M=0$ for every $M\in A\m$.
	\end{Lemma}
	\begin{proof}
		Since $A$ is hereditary, then $\tau=D\Ext_A^1(-, A)$. Thus $\tau$ is left exact and $F:=\tau^m$ is also left exact. Let $M\in A\m$. Then $0\rightarrow F\rad M\rightarrow FM\rightarrow F\top M=0$ is exact. By induction on the composition length of modules, we infer that $FM=0$.
	\end{proof}

Under our convention, if $A=kQ$, from $\Hom_A(Ae_j, Ae_i)\cong e_jAe_i$ we get that there exists an irreducible morphism $Ae_j\rightarrow Ae_i$ in $A\proj$ if and only if there exists an arrow $i\rightarrow j\in Q_1$. Hence, $\Gamma(A\proj)=Q^{op}$.
It is well known that for hereditary algebras, we can read the original quiver $Q$ directly inside the Auslander-Reiten quiver of $A$. This is due to the following.

	\begin{Lemma}\label{irreducibleprojectives}
	Let $A$ be a hereditary algebra. Given $X, Y\in A\proj$, a map $f\in \Hom_A(X, Y)$ is irreducible in $A\proj$ if and only if it is is irreducible in $A\m$.
\end{Lemma}
\begin{proof}
	If $f$ is irreducible in $A\m$, then it is irreducible in $A\proj$ by Lemma \ref{irreduciblefullsubcategories}. Assume that $f$ is irreducible in $A\proj$. By Lemma \ref{irreduciblefullsubcategories}, $f$ is not a split monomorphism nor a split epimorphism in $A\m$. Assume that $f$ admits a factorisation $f=h\circ g$ for some $g\in \Hom_A(X, Z)$ and $h\in \Hom_A(Z, Y)$. Since $A$ is hereditary and $\Im(h)\subset Y$, $\Im(h)\in A\proj$, moreover we can write $h=i\circ h'$ where $i$ is the inclusion $\Im(h)\subset Y$ and $h'\in \Hom_A(Z, \Im(h))$. Since $f$ is irreducible in $A\proj$ either $i$ is a split epimorphism or $h'\circ g$ is a split monomorphism. If the latter occurs, then $g$ is a split monomorphism, otherwise $\Im(h)=Y$ and so $h$ would split since $Y\in A\proj$. Thus, $f$ is irreducible in $A\m$.
\end{proof}

\subsection{Quasi-hereditary structures}\label{Quasi-hereditary structures}

 Quasi-hereditary algebras were introduced in \cite{MR961165}.  Standard references for basic properties on quasi-hereditary algebras are \cite{MR961165, Dlab1989d, MR1211481, MR1128706, PS88} and \cite[A]{MR1284468}.

 We will use the module theoretical definition of a quasi-hereditary algebra.
Let $A$ be a finite-dimensional algebra and $\L=(\L, \leq)$ a finite poset indexing a complete set of simple $A$-modules $\{S(\l)\colon \l\in \L \}$.
Denote by $P(\l)$ the projective cover of $S(\l)$ and by $I(\l)$ the injective hull of $S(\l)$ for $\l\in \L$. In particular, $I(\l)=DP_{op}(\l)$, where $P_{op}(\l)$ is the projective cover of $DS(\l)$ as right $A$-module.

The standard module $\St(\lambda)$ is the maximal quotient of $P(\lambda)$ whose composition factors are of the form $S(\mu)$ with $\mu\leq \lambda$. The costandard module $\Cs(\lambda)$ is the maximal submodule of $I(\lambda)$ whose composition factors are of the form $\mu\leq \lambda$. In particular, $\Cs(\l)=D\St_{A^{op}}(\l)$ for every $\l\in \L$. Whenever we want to emphasize the underlying algebra,  we write $\St_A(\lambda)$ instead of just $\St(\lambda)$. We use the analogous notation $\Cs_A(\l)$ for costandard modules.

Given a finite set $\Theta$ of objects in $A\m$, we denote by $\mathcal{F}(\Theta)$ the full subcategory of $A\m$ whose modules $M$ admit a filtration $0=F_m\subset F_{m-1}\subset \cdots\subset F_0=M$ such that each $F_i/F_{i+1}\in \Theta$ for each $0\leq i\leq m-1$.

\begin{Def}
	A finite-dimensional algebra $(A, \Lambda, \leq)$ is called \emph{quasi-hereditary} if for all $\lambda\in \L$:
	\begin{enumerate}[(1)]
		\item $\ker (P(\lambda)\rightarrow \St(\l))\in \mathcal{F}(\{\St(\mu)\colon \mu>\lambda\})$;
		\item $[\St(\lambda):S(\lambda)]=1$.
	\end{enumerate}
\end{Def}
It is also common to say that if (1) and (2)  hold, then $A$ admits a quasi-hereditary structure (or a highest weight structure). 
The condition (2) is equivalent to requiring that $\End_A(\St(\lambda))=k$ for all $\l\in \L$ (see \cite{MR1211481}). Moreover if $(A, \Lambda, \leq)$ is quasi-hereditary, then so is $(A^{op}, \Lambda, \leq)$ is quasi-hereditary and its standard modules are the duals of the costandard $A$-modules (see for instance \citep[Theorem 1]{MR1211481}).

We will write $\St_A$ and $\Cs_A$ (or just $\St$ or $\Cs$ when no ambiguity of $A$ exists) to denote the sets $\{\St(\l)\colon \l\in \L\}$ and $\{\Cs(\l)\colon \l\in \L\}$, respectively.

\begin{Remark}
	If $(A, \L, \leq)$ is quasi-hereditary, then there exists a total order $\preceq$ on $\L$ refining $\leq$ such that the corresponding standard modules and costandard modules are unchanged (see for instance \citep[Proposition 1.3]{zbMATH04150432} and \cite{MR961165}).
\end{Remark}

Examples of quasi-hereditary algebras include hereditary algebras and algebras of global dimension at most two. Further, hereditary algebras are quasi-hereditary with respect to any total order on the set of non-isomorphic classes of simple modules (see \cite{Dlab1989d, zbMATH04150432}). Further, hereditary algebras belong to the special class of quasi-hereditary algebras known as strongly quasi-hereditary algebras (see \citep[Theorem 4.1]{zbMATH07164450}). 

\subsubsection{The exact structure of $\mathcal{F}(\St)$ and $\mathcal{F}(\Cs)$} In this paragraph, we assume that $(A, \Lambda, \leq)$ is a quasi-hereditary algebra. 

For quasi-hereditary algebras the subcategories $\mathcal{F}(\Cs)$ and $\mathcal{F}(\St)$ are Krull-Schmidt categories and determine each other, in the sense that $\mathcal{F}(\St)={}^{\perp_1} \mathcal{F}(\Cs)$ and $\mathcal{F}(\Cs)=\mathcal{F}(\St)^{\perp_1}$ (see for instance \citep[Theorem 1]{MR1211481}). Moreover, $\mathcal{F}(\St)={}^{\perp} \mathcal{F}(\Cs)$ and $\mathcal{F}(\Cs)=\mathcal{F}(\St)^{\perp}$.  So it follows that $\mathcal{F}(\St)$ is a resolving subcategory of $A\m$, that is, it contains all projective modules, it is closed under extensions, closed under direct sums and summands and closed under kernels of epimorphisms.

However, $\mathcal{F}(\St)$ is not necessarily abelian because in general it is not closed under cokernels of monomorphisms. 
Being extension closed, it admits a canonical exact structure.
Indeed, both $\mathcal{F}(\St)$ and $\mathcal{F}(\Cs)$ inherit exact structures from the ambient abelian category $A\m$, with conflations given by the short exact sequences in $A\m$ whose terms belong to the respective subcategory. 

Since $\mathcal{F}(\St)$ is resolving, the objects $X\in \mathcal{F}(\St)$ in $\mathcal{F}(\St)\cap {}^{\perp_1} \mathcal{F}(\St)$ are the projective modules (see for instance \citep[Theorem 3 and Corollary 2]{MR1128706}).
This means that the projective objects in $\mathcal{F}(\St)$ are the projective modules over $A$. Dually, the injective objects in $\mathcal{F}(\Cs)$ are the injective modules.

The injective objects in $\mathcal{F}(\St)$ are the objects in $\mathcal{F}(\St)\cap \mathcal{F}(\St)^{\perp 1}=\mathcal{F}(\St)\cap \mathcal{F}(\Cs)$ and the projective objects in $\mathcal{F}(\Cs)$ are the objects in $\mathcal{F}(\St)\cap \mathcal{F}(\Cs)$.

In \cite{MR1128706}, it was proved that both the categories $\mathcal{F}(\St)$ and $\mathcal{F}(\Cs)$ are functorially finite and thus by \cite{zbMATH03715763} they have Auslander-Reiten sequences.

\subsubsection{The characteristic tilting module and Ringel duality} \label{The characteristic tilting module and Ringel duality}

By a \emph{tilting} $A$-module $M$, we mean a module $M$ with finite projective dimension such that  $\Ext_A^i(M, M)=0$ for all $i>0$ and there is an exact sequence $0\rightarrow A\rightarrow M_1 \rightarrow \cdots \rightarrow M_t$ with $M_i\in \add M$ for all $i\geq 1$ and some $t\geq 0$.
The injective objects in the exact structure of $\mathcal{F}(\St)$ are the modules in the additive closure of a very important tilting module.

\begin{Prop}
	Let $(A, \L, \leq)$ be a quasi-hereditary algebra.  Then there are unique indecomposable $A$-modules $T(\l)$ together with exact sequences
	\begin{align*}
		0\rightarrow  \St(\l)\rightarrow T(\l)\rightarrow X(\l)\rightarrow 0, \quad X(\l)\in \mathcal{F}(\{\St(\mu)\colon \mu<\l, \ \mu\in \L\})\\
		0\rightarrow Y(\l)\rightarrow T(\l)\rightarrow \Cs(\l)\rightarrow 0, \quad Y(\l)\in \mathcal{F}(\{\Cs(\mu)\colon \mu<\l, \ \mu\in \L \}).
	\end{align*}
	Furthermore, for $T=\bigoplus_{\l\in \L} T(\l)$, $\add T=\mathcal{F}(\St)\cap \mathcal{F}(\Cs)$. 
\end{Prop}
\begin{proof}
	See for example Proposition 2 of \cite{MR1128706}.
\end{proof} The module $T$ is known as the (basic) \emph{characteristic tilting module}. This module completely determines both $\mathcal{F}(\St)$ and $\mathcal{F}(\Cs)$.
\begin{Prop}\label{characteristictilting}	Let $(A, \L, \leq)$ be a quasi-hereditary algebra and $T$ be the characteristic tilting module. The following assertions hold.
	\begin{multicols}{2}
		\begin{enumerate}[(a)]
			\item $\mathcal{F}(\St)=\{M\in A\m \colon  \Ext_A^{i>0}(M, T)=0 \}$;
			\item $\mathcal{F}(\Cs)= \{M\in A\m \colon  \Ext_A^{i>0}(T, M)=0 \}.$
		\end{enumerate}
	\end{multicols}
	\begin{enumerate}[(c)]
		\item 	 The algebra $\End_A(T)^{op}$ is quasi-hereditary together with the standard modules $\Hom_A(T, \Cs(\l))$, $\l\in \L$, and the poset $\L$ equipped with the reversed order.  Moreover, the functor $\Hom_A(T, -)$ induces an exact equivalence between $\mathcal{F}(\Cs)$ and $\mathcal{F}(\Hom_A(T, \Cs))$
	\end{enumerate}
\end{Prop}
\begin{proof}
	See for example Corollary 4 and Theorem 6 of \cite{MR1128706}.
\end{proof}

Denote by $\L^{op}$ the poset $\L$ equipped with the reversed order.
The algebra $R_A:=\End_A(T)^{op}$ is called the \emph{Ringel dual} of $(A, \{\St(\l)_{\l\in \L}\})$ and the functor $\Hom_A(T, -)$ is called the \emph{Ringel dual functor}. By \cite{MR1128706}, we can see that the Ringel dual functor sends costandard modules over $A$ to standard modules over $R_A$, it sends tilting modules over $A$ to projective modules over $R_A$ and it sends injective modules over $A$ to tilting modules over $R_A$.
In particular, the Ringel dual of the Ringel dual is Morita equivalent to the original algebra.
We say that a quasi-hereditary algebra $(A, \L, \leq)$ is \emph{Ringel self-dual} if 
there exists an equivalence of categories $F\colon R_A\m\rightarrow A\m$ and a bijection $\psi\colon \L^{op}\rightarrow \L$ such that $F\St_R(\l)=\St(\psi(\l))$ for all $\l\in \L^{op}$, where $\St_R(\l)$ stands for the standard module $\Hom_A(T, \Cs(\l))$. Equivalently, by the standardisation theorem \citep[Theorem 2]{MR1211481} and Proposition \ref{characteristictilting}(c), a quasi-hereditary algebra is Ringel self-dual if there exists an exact equivalence between $\mathcal{F}(\St)$ and $\mathcal{F}(\Cs)$.

\begin{Lemma}
	Let $(A, \L, \leq)$ be a quasi-hereditary algebra that it is Ringel self-dual. Then $(A^{op}, \L, \leq)$ is also Ringel self-dual.
\end{Lemma}
\begin{proof}
	Since $(A, \L, \leq)$ is Ringel self-dual, there is an exact equivalence $\Psi\colon \mathcal{F}(\St_A)\rightarrow \mathcal{F}(\Cs_A)$. The duality $D$ restricts to contravariant exact equivalences $\mathcal{F}(\Cs_A)\rightarrow \mathcal{F}(\St_{A^{op}})$ and $\mathcal{F}(\Cs_{A^{op}})\rightarrow \mathcal{F}(\St_A)$. Consequently, $D\circ \Psi \circ D\colon \mathcal{F}(\Cs_{A^{op}})\rightarrow \mathcal{F}(\St_{A^{op}})$ is an exact equivalence. Thus, $(A^{op}, \L, \leq)$  is Ringel self-dual.
\end{proof}

\subsection{Torsion pairs} \label{Torsion pairs}
Torsion pairs in abelian categories were introduced by Dickson in \cite{zbMATH03223714} and have since become a standard tool in the representation theory of finite-dimensional algebras and, in particular, in tilting theory. For torsion pairs in module categories of finite-dimensional algebras, a standard reference is \cite{zbMATH02228448}. For further background see for instance \cite{zbMATH00867897, zbMATH05179140}.

A pair  $(\mathcal{A}, \mathcal{B})$ of full subcategories of $A\m$ is a \emph{torsion pair} if $\Hom_A(M, N)=0$ for all $M\in\mathcal{A}$ and $N\in \mathcal{B}$ and for every $A$-module $X$ there exists an exact sequence $0\rightarrow tX\rightarrow X\rightarrow X/tX\rightarrow 0$ with $tX\in \mathcal{A}$ and $X/tX\in \mathcal{B}$. In such a case, such exact sequence is unique. In a torsion pair  $(\mathcal{A}, \mathcal{B})$ the subcategory $\mathcal{A}$ is called the \emph{torsion class} while $\mathcal{B}$ is called the \emph{torsion-free class}.

We say that a torsion pair $(\mathcal{A}, \mathcal{B})$ in $A\m$ is \emph{split} (or a \emph{splitting torsion pair}) if each indecomposable $A$-module lies either in $\mathcal{A}$ or in $\mathcal{B}$, equivalently if $\Ext_A^1(N, M)=0$ for all $M\in \mathcal{A}$ and $N\in \mathcal{B}$. 

\begin{Lemma}\label{dualtorsionpair}
	Let $A$ be a finite-dimensional algebra. Then a pair of subcategories $(\mathcal{X}, \mathcal{Y})$ is a torsion pair in $A\m$ if and only if the pair $(D\mathcal{Y}, D\mathcal{X})$ is a torsion pair in $A^{op}\m$. In particular, $(\mathcal{X}, \mathcal{Y})$ is a split torsion pair if and only $(D\mathcal{Y}, D\mathcal{X})$ is. 
\end{Lemma}
\begin{proof}
	It follows as direct application of the equality $\Hom_A(M, N)\cong \Hom_{A^{op}}(DN, DM)$ in the definition of torsion pair. The second statement can be deduced for example by using the equality $\Ext_A^1(M, N)\cong \Ext_{A^{op}}(DN, DM)$ together with \citep[VI, Proposition 1.7]{zbMATH02228448}.
\end{proof}

Given $M\in A\m$, we write \[\gen M=\{N\in A\m\colon \text{ there exists an epimorphism } M_0\twoheadrightarrow N \text{ with } M_0\in \add M\}\] to be the class of modules generated by $M$ and \[\cogen M=\{N\in A\m\colon \text{ there exists a monomorphism } N\hookrightarrow M_0 \text{ with } M_0\in \add M\}\] the class of modules cogenerated by $M$.

A module $T\in A\m$ is called a \emph{classical tilting module} if $T$ is a tilting $A$-module with projective dimension at most one. In particular, over hereditary algebras all tilting modules are classical tilting modules.

For a given classical tilting $A$-module, fix $B=\End_A(T)^{op}$ and consider the following full subcategories
\begin{align*}
	\mathcal{T}(T)&=\{M\in A\m\colon \ \Ext_A^1(T, M)=0 \}, 
	& 
	\mathcal{S}(T)&=\{M\in A\m\colon \ \Hom_A(T, M)=0 \},\\
	\mathcal{Y}(T)&=\{Y\in B\m \colon \ \Ext_B^1(Y, DT)=0 \}, &
	\mathcal{X}(T)&=\{X\in B\m \colon \ \Hom_B(X, DT)=0 \},\\
	\mathcal{W}(T)&=\{X\in A\m \colon \ \Ext^1_A(X, T)=0 \}, &
	\mathcal{U}(T)&=\{X\in A\m \colon \ \Hom_A(X, T)=0 \}.
\end{align*}

We can formulate these categories also in terms of subcategories generated or cogenerated by a certain module. For instance, $\mathcal{T}(T)=\gen T$ and $\mathcal{S}(T)=\cogen \tau T$ (see for example \citep[VI, Theorem 2.5]{zbMATH02228448}).

The tilting theorem of Brenner and Butler \cite{zbMATH03697330} clarifies how these subcategories relate with each other, and in particular, provides a rich source of examples of torsion pairs arising from classical tilting modules.

\begin{Theorem}[Brenner-Butler] \label{BrennerButler}
	Let $A$ be a finite-dimensional algebra and $T$ a classical tilting module. Fix $B=\End_A(T)^{op}$. Then, the following assertions hold.
	\begin{enumerate}[(a)]
		\item The pair $(\mathcal{T}(T), \mathcal{S}(T))$ is a torsion pair in $A\m$ and $(\mathcal{X}(T), \mathcal{Y}(T))$ is a torsion pair in $B\m$.
		\item The functors $\Hom_A(T, -)$ and $T\otimes_B -$ induce quasi-inverse equivalences between $\mathcal{T}(T)$ and $\mathcal{Y}(T)$;
		\item The functors $\Ext_A^1(T, -)$ and $\Tor_1^B(T, -)$ induce quasi-inverse equivalences between $\mathcal{S}(T)$ and $\mathcal{X}(T)$. 
	\end{enumerate}
\end{Theorem}
\begin{proof}
	See for example \citep[VI Theorem 3.8]{zbMATH02228448}.
\end{proof}

We are interested in these subcategories of $\End_A(T)^{op}\m$ mainly because of the following application of the tilting theorem of Brenner and Butler, due to Happel and Ringel \citep[Theorem 6.3]{zbMATH03792379}. It shows that, when $A$ is hereditary, the module category of the endomorphism algebra of a tilting modules is completely separated into two classes. We complement this with a related torsion pair in $A\m$.

\begin{Theorem}\label{separation}
	Let $A$ be a hereditary finite-dimensional algebra. If $T$ is a tilting module of $A$, then the following assertions hold.
	\begin{enumerate}[(a)]
		\item The pair $(\mathcal{X}(T), \mathcal{Y}(T))$ is a split torsion pair in $\End_A(T)^{op}\m$.
		\item The pair $(\mathcal{U}(T), \mathcal{W}(T))=(\gen \tau^{-1}T, \cogen T)$ is a torsion pair in $A\m$.
	\end{enumerate} 
\end{Theorem}
\begin{proof}
	For (a), see for example \citep[VI Theorem 5.6, Corollary 5.7]{zbMATH02228448}. For (b), observe that $A^{op}$ is again hereditary and so $DT$ is a classical tilting module of $A^{op}$. By Theorem \ref{BrennerButler}(a) and Lemma \ref{dualtorsionpair}, it follows that $(\mathcal{U}(T), \mathcal{W}(T))=(D\mathcal{S}(DT), D\mathcal{T}(DT))$ is a torsion pair in $A\m$. The pairs $(\mathcal{U}(T), \mathcal{W}(T))$ and $(\gen \tau^{-1}T, \cogen T)$ coincide because $D\gen M=\cogen DM$ for every module $M\in A^{op}\m$.
\end{proof}

For hereditary algebras, Ringel duality fits particularly well into the framework of the Brenner-Butler theorem. In such a case, the characteristic tilting module $T$ is a classical tilting module and $\mathcal{F}(\Cs)$ is the torsion class $\mathcal{T}(T)$ while $\mathcal{F}(\St_{\End_A(T)^{op}})$ is the torsion-free class $\mathcal{Y}(T)$. Thus, the exact equivalence between these subcategories induced by Ringel duality is precisely the corresponding Brenner-Butler equivalence presented in (b)  of Theorem \ref{BrennerButler}. Moreover, in this setting, Ringel duality is complemented 
by the torsion-theoretic structure provided by Brenner and Butler.

\section{Properties of Ringel self-dual hereditary algebras}\label{Properties}

In this section, we investigate structural properties of Ringel self-dual hereditary algebras. Our main aim is to develop criteria for determining whether a hereditary algebra of finite representation type is Ringel self-dual without explicitly computing its Ringel dual. The key ingredient is a precise description of the characteristic tilting module, which culminates in the proof of Theorem \ref{maintheorem} at the end of the section.
 
 \subsection{Two split torsion pairs associated with the characteristic tilting module}\label{Two split torsion pairs associated with the characteristic tilting module}

The aim is to understand how Ringel self-duality constrains the position of the characteristic tilting module inside the module category of a hereditary algebra. Our strategy begins by considering the torsion pairs naturally associated with the characteristic tilting module and showing that, in the Ringel self-dual case, these torsion pairs are split. This will allow us to  show Ringel self-duality forces every indecomposable module to lie in $\mathcal{F}(\St)$, in $\mathcal{F}(\Cs)$ or both, and in the finite type case, forces the characteristic tilting module to be a complete slice in the sense of Happel and Ringel.

\begin{Prop}\label{prop}
	Let $A$ be a connected hereditary algebra, $\L$ a finite poset such that $(A, \L, \leq)$ is Ringel self-dual and $T$ its (basic) characteristic tilting module. Then the following assertions hold.
	\begin{enumerate}[(a)]
		\item The pairs $(\mathcal{F}(\Cs), \cogen \tau T)$ and $(\gen \tau^{-1}T, \mathcal{F}(\St))$ are split torsion pairs in $A\m$. Moreover, $\gen \tau^{-1}T=\{X\in A\m\colon \Hom_A(X, T)=0\}$, $\cogen \tau T=\{X\in A\m\colon \Hom_A(T, X)=0\}$, $\gen T=\mathcal{F}(\Cs)$ and $\cogen T=\mathcal{F}(\St)$.
		\item A simple $A$-module $S$ is standard (resp. costandard) if and only if $S\in \operatorname{soc}T$ (resp. $S\in \operatorname{top} T$). 
		\item $\cogen \tau T\subset \mathcal{F}(\St)$ and $\gen \tau^{-1}T\subset \mathcal{F}(\Cs)$. In particular, $\gen \tau^{-1}T\cap \cogen \tau T=\{0\}$ and every indecomposable module $X$ not in $\add T$ is either in $\mathcal{F}(\Cs)$ or in $\mathcal{F}(\St)$.
		\item Let $M\in A\m$ be an indecomposable module. If $\Hom_A(M, A)\neq 0$, then $M\in A\proj$.		
		\item If $\mathcal{F}(\St)=A\m$ or $\mathcal{F}(\Cs)=A\m$, then $A$ is semi-simple.
	\end{enumerate}
\end{Prop}
\begin{proof}
By Theorem \ref{BrennerButler}, $(\mathcal{T}(T), \mathcal{S}(T))$ is a torsion pair of $A\m$ and $\gen T=T(T)=\mathcal{F}(\Cs)$.  By Theorem \ref{separation}, $(\gen \tau^{-1}T, \mathcal{F}(\St))$ is a torsion pair in $A\m$. 
 By Theorem \ref{BrennerButler}, $(\mathcal{X}(T), \mathcal{Y}(T))$ is a torsion pair in $B\m$, where $B=\End_A(T)^{op}$.
	 By \citep[VI, Theorem 5.6]{zbMATH02228448}, $(\mathcal{X}(T), \mathcal{Y}(T))$ is a split torsion pair in $B\m$.
 
  $B$ is the Ringel dual of $A$ and so $DT\cong \Hom_A(T, DA)$ is the characteristic tilting module of $B$.
 Since $A$ is Ringel self-dual, then there exists an exact equivalence of categories $F\colon B\m\rightarrow A\m$ that restricts to an equivalence of to the respective exact categories of modules having a filtration by standard modules. So $\add F(DT)=\add T$ and by applying $F$ to the torsion pair $(\mathcal{X}(T), \mathcal{Y}(T))$ we get the torsion pair
 $(\gen \tau^{-1}T, \mathcal{F}(\St))$, that is, $(\gen \tau^{-1}T, \mathcal{F}(\St))$ is a split torsion pair in $A\m$. Since $\gldim B\leq 1$, again by \citep[VI, Theorem 5.6]{zbMATH02228448} we obtain that $(\mathcal{F}(\Cs), \cogen \tau T)$ is split. So (a) holds.
 
The statement (b) is a direct consequence of (a).

Since the torsion pair $(\mathcal{F}(\Cs), \cogen \tau T)$ is split, it follows that $\Ext_A^1(M, N)=0$ for all $M\in \cogen \tau T$ and $N\in \mathcal{F}(\Cs)$. Hence $\cogen \tau T\subset \mathcal{F}(\St)$. Similarly, we get that $\gen \tau^{-1}T\subset \mathcal{F}(\Cs)$. So, $\cogen \tau T\cap \gen \tau^{-1}T\subset \add T$ but since $\Hom_A(X, T)=0$ for every $X\in \cogen \tau^{-1} T$ we get that $\cogen \tau T\cap \gen \tau^{-1}T=\{0\}$. The last statement of (c) follows from (a) together with $\cogen \tau T\subset \mathcal{F}(\St)$.

To show (d). Let $M\in A\m$ be an indecomposable module. By (a), $M$ lies either in $\gen \tau^{-1}T$ or in $\mathcal{F}(\St)$. Since $A\in \mathcal{F}(\St)$ then since $\Hom_A(M, A)\neq 0$ we get that $M\in \mathcal{F}(\St)$. Let $\Psi$ be an exact equivalence $\mathcal{F}(\St)\rightarrow \mathcal{F}(\Cs)$. Then $0\neq \Hom_A(M, A)\cong \Hom_A(\Psi M, \Psi A)\cong \Hom_A(\Psi M, T)$. So, $\Psi M$ is indecomposable and $\Psi M\notin \gen \tau^{-1}T$, hence $\Psi M\in \mathcal{F}(\St)$. Therefore, $\Psi M\in \add T$. Since the objects in $\add T$ are the projective objects of $\mathcal{F}(\Cs)$ we obtain that $M\in A\proj$.
	
It remains to prove (e). Assume that $\mathcal{F}(\Cs)=A\m$, the other case is analogous. So $\mathcal{F}(\St)={}^\perp \mathcal{F}(\Cs)=A\proj$ so $\mathcal{F}(\St)$ would have exactly $n=|A|$ indecomposable objects. Since there exists an exact equivalence between $\mathcal{F}(\Cs)$ and $\mathcal{F}(\St)$ we obtain that $A\m=\mathcal{F}(\Cs)$ would have exactly $n$ indecomposable objects, namely the $n$ simple modules, so $A$ is semi-simple.
\end{proof}

As a direct consequence, for a Ringel self-dual hereditary algebra $A$, every indecomposable direct summand of $\operatorname{soc}(A)$ is a standard module. We recall that, although the torsion pairs induced by the characteristic tilting module in Proposition \ref{prop} (a) exist for arbitrary hereditary algebras, they need not be split in general.

\begin{Example}
	Let $Q$ be the quiver $1\rightarrow 2\rightarrow 3$ and $A$ the path algebra of $Q$ over a field $k$. Consider the partial order $3<1<2$, then the standard modules are $\St(2)=P(2)$, $\St(1)=S(1), \St(3)=S(3)$ and so the characteristic tilting module is $T=S(1)\oplus S(3)\oplus P(1)$. Thus, $\tau T=\tau^{-1} T=S(2)$ and one can see that $P(2)\notin \mathcal{F}(\Cs)$ and $I(2)\notin \mathcal{F}(\St)$ so the torsion pairs $(\mathcal{F}(\Cs), \cogen \tau T)$ and $(\gen \tau^{-1}T, \mathcal{F}(\St))$ are not split.
\end{Example}

One consequence of Proposition \ref{prop} is that the module category of a Ringel self-dual hereditary algebra decomposes into three pairwise disjoint parts: $\gen \tau^{-1}T, \add T, \cogen \tau T.$ 
This raises the question of how the simple modules are distributed among them. Define 
$\mathcal{S}_0=\{\lambda\in \L\colon S(\l)=\St(\l)\neq \Cs(\l) \}$,
 $\mathcal{T}_0=\{\lambda\in \Lambda\colon S(\lambda)=\St(\lambda)=\Cs(\l)\}$, $\mathcal{U}_0=\{\lambda\in \Lambda\colon S(\lambda)=\Cs(\lambda)\neq \St(\l)\}$.
 \begin{Cor}
 		Let $A$ be a connected hereditary algebra, $\L$ a finite poset such that $(A, \L, \leq)$ is Ringel self-dual and $T$ its characteristic tilting module. Let $\l\in \L$.Then the following assertions hold.
 		\begin{enumerate}[(a)]
 			\item  $S(\l)\in \cogen \tau T$ if and only if $\lambda\in \mathcal{S}_0$.
 			\item  $S(\l)\in \add T$ if and only if $\lambda\in \mathcal{T}_0$.
 			\item  $S(\l)\in \gen \tau^{-1} T$ if and only if $\lambda\in \mathcal{U}_0$.
 			\item  If $\lambda\in \mathcal{U}_0$ (resp. $\lambda \in \mathcal{S}_0$), then $S(\lambda)$ is not projective (resp. injective).
 		\end{enumerate}
 \end{Cor}
\begin{proof}
	Observe that $S(\lambda)\in \mathcal{F}(\St)$ if and only if $S(\lambda)=\St(\lambda)$. Indeed, if $S(\lambda)\in \mathcal{F}(\St)$, then $S(\lambda)$ surjects to one standard module, so they have the same top and since $S(\lambda)$ is a simple such surjection must be a bijection. Thus by definition we get in such a case $S(\lambda)=\St(\lambda)$.

	 Similarly, $S(\lambda)\in \mathcal{F}(\Cs)$ if and only if $S(\l)=\Cs(\l)$. Thus (b) follows.  	Also by Proposition \ref{prop}, (a) follows. The statement (c) is analogous.
	
	If $\lambda\in \mathcal{U}_0$, then $S(\lambda)\notin \mathcal{F}(\St)$, so it cannot be projective. The other case in (d) is analogous.
\end{proof}

Another direct consequence of Proposition \ref{prop} concerns the exact structure of $\mathcal{F}(\St)$. Indeed, for Ringel self-duality hereditary algebras the Auslander-Reiten sequences of $\mathcal{F}(\St)$ are exactly the absolute ones in $A\m$ with terms in $\mathcal{F}(\St)$.

\begin{Cor}\label{closedundertau}
	Let $A$ be a connected hereditary algebra, $\L$ a finite poset such that $(A, \L, \leq)$ is Ringel self-dual. Then the following assertions hold.
	\begin{enumerate}[(a)]
		\item Let $Z\in\mathcal{F}(\St)$ be indecomposable and non-projective.  Then the
		Auslander--Reiten sequence
		$0\to\tau Z\to E\to Z\to 0$ of $A\m$ has all its terms in
		$\mathcal{F}(\St)$ and is an Auslander--Reiten sequence in $\mathcal{F}(\St)$.
		Conversely, every Auslander--Reiten sequence in $\mathcal{F}(\St)$ arises in
		this way.
		\item Let $X, Y\in \mathcal{F}(\St)$. A map $f\in \Hom_A(X, Y)$ is irreducible in $A\m$ if and only if it is irreducible in $\mathcal{F}(\St)$.
		\item  Let $X\in\mathcal{F}(\Cs)$ be indecomposable and non-injective.  Then the
		Auslander--Reiten sequence $0\to X\to E\to\tau^{-}X\to 0$ of $A\m$ has all
		its terms in $\mathcal{F}(\Cs)$ and is an Auslander--Reiten sequence in
		$\mathcal{F}(\Cs)$; every Auslander--Reiten sequence in $\mathcal{F}(\Cs)$
		arises in this way.
		\item Let $X, Y\in \mathcal{F}(\Cs)$. A map $f\in \Hom_A(X, Y)$ is irreducible in $A\m$ if and only if it is irreducible in $\mathcal{F}(\Cs)$.
	\end{enumerate}
\end{Cor}
\begin{proof} We only prove (a). The statement (c) is dual to (a) while statement (d) is dual to (b).
	By Proposition \ref{prop}, the pair $(\gen \tau^{-1}T, \mathcal{F}(\St))$ is a split torsion pair. So $\mathcal{F}(\St)$ is closed under $\tau$ (see for instance \citep[VI, Proposition 1.7]{zbMATH02228448}).
	Let $Z$ be as in the
	statement and $\varepsilon:0\to\tau Z\xrightarrow{f}E\xrightarrow{g}Z\to 0$ the
	Auslander--Reiten sequence of $A\m$.  By the above $\tau Z\in\mathcal{F}(\St)$, and
	$\mathcal{F}(\St)$ is closed under extensions, so $E\in\mathcal{F}(\St)$ and
	$\varepsilon$ is a conflation of $\mathcal{F}(\St)$ and it is non-split.  Let
	$C\in\mathcal{F}(\St)$ and let $t\colon C\to Z$ be a morphism which is not a
	split epimorphism in $\mathcal{F}(\St)$.  Since $\mathcal{F}(\St)\subseteq A\m$ is full,
	a split epimorphism of $t$ in $A\m$ would already lie in $\mathcal{F}(\St)$; hence $t$ is not
	a split epimorphism in $A\m$, and as $g$ is right almost split in $A\m$ there is
	$s\in\Hom_A(C,E)$ with $t=gs$.  Again by fullness, $s$ is a morphism of
	$\mathcal{F}(\St)$.  Thus $g$ is right almost split in $\mathcal{F}(\St)$, and
	dually $f$ is left almost split in $\mathcal{F}(\St)$.
	
	Let
	$\eta:0\to X'\to Y'\to Z'\to 0$ be an Auslander--Reiten sequence in
	$\mathcal{F}(\St)$.  Then $Z'$ is indecomposable and not a projective object in
	$\mathcal{F}(\St)$. Since the projectives objects in $\mathcal{F}(\St)$ are exactly
	the projective $A$-modules, $Z'$ is a non-projective $A$-module.
	By the previous paragraph the Auslander--Reiten sequence of $A\m$ ending at $Z'$ is
	an Auslander--Reiten sequence in $\mathcal{F}(\St)$, so $\eta$ is isomorphic to it
	by the uniqueness of almost split conflations in the Krull--Schmidt exact category
	$\mathcal{F}(\St)$ (see for instance \cite[\S2.2]{zbMATH01224795} or \cite{zbMATH03715763}).
	
	We shall now prove (b). It is enough to assume that $f\in \Hom_A(X, Y)$ is irreducible in $\mathcal{F}(\St)$, the other case follows from  Lemma \ref{irreduciblefullsubcategories}. By Lemma \ref{irreduciblefullsubcategories}, $f$ is neither a split monomorphism nor a split epimorphism in $A\m$.
	If $Y\in A\proj$, then since $A\proj\subset \mathcal{F}(\St)$ the same argument as done in Lemma \ref{irreducibleprojectives} yields that $f$ is irreducible in $A\m$.
	So assume that $Y$ is not projective. Consider the Auslander-Reiten sequence $0\rightarrow \tau Y\rightarrow E\xrightarrow{g} Y\rightarrow 0$. By part (a), this exact sequence lies in $\mathcal{F}(\St)$. By the right almost split property of $g$ in $A\m$, $f$ factors as $f=g\circ h$ for some $h\in \Hom_A(X, E)$. Since $\mathcal{F}(\St)$ is full in $A\m$ this is a factorisation in $\mathcal{F}(\St)$ and consequently $h$ must be a split monomorphism because $f$ is irreducible in $\mathcal{F}(\St)$. It follows that $f$ is irreducible in $A\m$ (see for instance \citep[V, Theorem 5.3]{zbMATH00707210}).
\end{proof}

To further understand the symmetry that the module category of a Ringel self-dual hereditary algebra, we shall now focus on algebras of finite representation type.

\begin{Prop} \label{counting}
	Let $A$ be a connected hereditary algebra of finite representation type, $\L$ a finite poset such that $(A, \L, \leq)$ is Ringel self-dual and $T$ its characteristic tilting module. Then the following assertions hold.
		\begin{enumerate}[(a)]
		\item Given any indecomposable module $X$ in $A\m$, there exists exactly one element $\lambda\in \Lambda$ such that $T(\l)$ and $X$ are in the same $\tau$-orbit.
		\item If $X_0\rightarrow X_1 \rightarrow \cdots \rightarrow X_t$ is a chain of non-zero maps and indecomposable modules, and $X_0, X_t\in \add T$, then $X_i\in \add T$ for all $i=0, \ldots, t$. 
		\item The number of non-isomorphic indecomposable $A$-modules is equal to $|A|+2|\cogen \tau T|=|A|+2|\gen \tau^{-1}T|$.
	\end{enumerate}
\end{Prop}
\begin{proof}
		Statements (a) and (b) follow from \citep[Section 7]{zbMATH03792379} and \citep[IX, Lemma 1.1]{zbMATH02228448}.
	
	By Proposition \ref{prop}(a), \begin{align*}
		|A\m|=|\mathcal{F}(\Cs)|+|\cogen \tau T|=|\mathcal{F}(\St)|+|\gen \tau^{-1} T|.
	\end{align*} Since the categories $\mathcal{F}(\Cs)$ and $\mathcal{F}(\St)$ are equivalent we get $|\mathcal{F}(\Cs)|=|\mathcal{F}(\St)|$ and thus $|\cogen \tau T|=|\gen \tau^{-1} T|.$ Also every indecomposable object in $\mathcal{F}(\Cs)$ is either in $\gen \tau^{-1}T$ or in $\add T$ by Proposition \ref{prop}(a). So,
	\begin{align*}
		|A\m|=|\mathcal{F}(\Cs)|+|\cogen \tau T|=|\gen \tau^{-1}T|+|T|+|\cogen \tau T|=|A|+2|\cogen \tau T|.
	\end{align*}
\end{proof}

A module over a hereditary algebra satisfying the properties (a) and (b) of Proposition \ref{counting} is called \emph{complete slice} in \citep[Section 7]{zbMATH03792379} (see also \citep[IX, Lemma 1.1]{zbMATH02228448}).

\subsection{Proof of the main result}\label{Proof of the main result}

The characteristic tilting module being a complete slice gives us information on the size of the $\tau$-orbits of the Auslander-Reiten quiver of the Ringel self-dual algebra. Let $\l\in \L$ and let $l_\lambda$ be the size of the $\tau$-orbit of $P(\lambda)$. Since $T$ is a complete slice, each $\tau$-orbit is of the form $\{P(\l), \ldots, T_\lambda, \ldots, I_\lambda\}$, where $I_\lambda$ is injective and $T_\lambda$ is an indecomposable direct summand of the characteristic tilting module.
As we will see next, this fact together with Corollary \ref{closedundertau} help us to locate the characteristic tilting module in the Auslander-Reiten quiver of a Ringel self-dual hereditary algebra of finite representation type.

\begin{Lemma}\label{QuiveraddT}
	Let $k$ be an algebraically closed field and let $Q$ be a connected quiver whose underlying graph is a Dynkin diagram of type $A$, $D$ or $E$ and vertex set $Q_0=\{1, \ldots, n\}$ and fix $A=kQ$. Assume that $(A, Q_0, \leq)$ is Ringel self-dual for a partial order $\leq$ on $Q_0$ and let $T$ be the basic characteristic tilting module of $(A, Q_0, \leq)$. 
Then the following assertions hold.
	\begin{enumerate}[(a)]
		\item The Gabriel quiver of $\End_A(T)^{op}\cong A$ is obtained by following the arrows between the direct summands of $T$ in the Auslander-Reiten quiver of $A$.
		\item Let $i\in Q_0$, $p_i, q_i\in \mathbb{N}_0$ such that $\tau^{-p_i}P(i)\in \add T$ and $\tau^{-p_i-q_i}P(i)$ is a non-zero injective module. Then there exists a graph automorphism $\sigma$ of the underlying Dynkin diagram of $Q$ with the following properties:
		\begin{enumerate}[(i)]
			\item For every $i\in Q_0$, $p_i=q_{\sigma(i)}$.
			\item If $\sigma^2=\id$ and $\sigma$  preserves the orientation of $i\rightarrow j\in Q_1$, then $p_i=p_j$.
			\item If $\sigma^2=\id$ and $\sigma$  reverses the orientation of $i\rightarrow j\in Q_1$, then $p_j=p_i+1$.
			\item If $\sigma^2=\id$, then all $\tau$-orbits have the same length.
		\end{enumerate}

	\end{enumerate}
\end{Lemma}
\begin{proof}
(a) is a direct consequence of $T$ being a slice module and so the irreducible maps (resp. the composition of irreducible maps) are the same in the additive categories $\add T$ and $A\m$, see for instance \citep[VIII, Theorem 3.5]{zbMATH02228448}).

Let $T_i$ be the indecomposable $\tau^{-p_i}(P(i)).$ By Corollary \ref{closedundertau}, $\mathcal{F}(\St)$ is closed under $\tau$ while $\mathcal{F}(\Cs)$ is closed under $\tau^{-1}$. Thus since $T_i$ is an injective object in $\mathcal{F}(\St)$ we get that $T_i, \tau T_i, \ldots, \tau^{p_i}T_i=P(i)$ are the modules that belong to the same $\tau$-orbit of $T_i$ while being in $\mathcal{F}(\St)$ and its length is $p_i+1$. Similarly, the objects in the $\tau$-orbit of $T_i$ that lie in $\mathcal{F}(\Cs)$ are $T_i, \tau^{-1} T_i, \ldots, \tau^{-q_i}T_i$ which has length $q_i+1$. Let $\Psi$ be an exact equivalence $\Psi\colon \mathcal{F}(\St)\rightarrow \mathcal{F}(\Cs)$. Thus, it sends projective objects in $\mathcal{F}(\St)$ to projective objects in $\mathcal{F}(\Cs)$, injective objects in $\mathcal{F}(\St)$ to injective objects in $\mathcal{F}(\Cs)$.
 Thus, there exists a permutation $\sigma\in S_n$ such that $\Psi$ sends $P(i)$ to $T_{\sigma(i)}$. 
 
 \begin{claim}
 	For every $i\in Q_0$, $\Psi T_i\cong \tau^{-p_i} T_{\sigma(i)}$. \label{claim1}
 \end{claim}
\begin{claimproof}
Let $i\in Q_0$. Assume first that $p_i=0$. Then $\Psi T_i\cong \Psi P(i)\cong T_{\sigma(i)}$. Suppose now that $p_i>0$. Suppose now that $p_i>0$.
Let $X\in \mathcal{F}(\St)$ and not projective. By Corollary \ref{closedundertau}, the Auslander-Reiten sequence $0\rightarrow \tau X\rightarrow E\rightarrow X\rightarrow 0$ lies in $\mathcal{F}(\St)$. Applying $\Psi$ we get the Auslander-Reiten sequence $0\rightarrow \Psi \tau X\rightarrow \Psi E\rightarrow \Psi X\rightarrow 0$ in $\mathcal{F}(\Cs)$. It is also an Auslander-Reiten sequence in $A\m$ thus $\Psi X\cong \tau^{-1} \Psi \tau X$. In particular, we can apply this successively to $X\in \{T_i, \ldots, \tau^{p_i-1} T_i\}$.
By induction on $r\in \{0, \ldots, p_i\}$ we get that $\Psi T_i\cong \tau^{-r}\Psi \tau^r T_i$. Indeed for $r<p_i$, the module $\tau^r T_i$ is not projective, so $\Psi \tau^r T_i\cong \tau^{-1}\Psi \tau^{r+1} T_i$ and $\Psi T_i\cong \tau^{-r}\Psi \tau^r T_i \cong \tau^{-r} \tau^{-1}\Psi \tau^{r+1} T_i$.  Taking $r=p_i$, we obtain 
\[  \Psi T_i\cong \tau^{-p_i}\Psi \tau^{p_i} T_i\cong \tau^{-p_i} \Psi P(i)\cong \tau^{-p_i} T_{\sigma(i)}. \qedhere \]
\end{claimproof}

\begin{claim}
	For every $i\in Q_0$, $p_i=q_{\sigma(i)}$ and $\Psi T_i\cong I_{\sigma(i)}$.
\end{claim}
\begin{claimproof}
	Let $i\in Q_0$. Since $T_i$ is an injective object in $\mathcal{F}(\St)$, $\Psi T_i$ is an injective $A$-module. By Claim \ref{claim1}, $\tau^{-p_i} T_{\sigma(i)}$ is injective, and by definition the only injective in the $\tau$-orbit of $T_{\sigma(i)}$ is $I_{\sigma(i)}$. Thus $\Psi T_i\cong I_{\sigma(i)}$. Also $q_i$ is, by definition, the unique integer such that $\tau^{-q_i}T_i$ is a non-zero injective module. Thus, $p_i=q_{\sigma(i)}$. 
\end{claimproof}

\begin{claim}
	Let $i\rightarrow j\in Q_1$ be an arrow. Then $[T_i]$ and $[T_j]$ are adjacent in the underlying graph of $\Gamma(\add T)$. Moreover, $p_j-p_i\in \{0, 1\}$ with $p_i=p_j$ if and only if there exists an irreducible map $T_j\rightarrow T_i$. \label{claim3}
\end{claim}
\begin{claimproof}
	Let $i\rightarrow j\in Q_1$, then there exists an irreducible map $P(j)\rightarrow P(i)$ in $A\m$. Further by considering the irreducible map $\tau^{-p_j}P(j)\rightarrow \tau^{-p_j}P(i)$ and $T$ being a complete slice we obtain that either $\tau^{-p_j}P(i)\in \add T$ or $\tau^{-p_j+1}P(i)\in \add T$ (see for instance \citep[VIII, Lemma 1.4]{zbMATH02228448}). Recall that in the Auslander-Reiten quiver of $A$, the arrows between the $\tau$-orbits of $P(i)$ and $P(j)$ are, whenever the corresponding modules are non-zero,
	of the form $\tau^{-n} P(j)\rightarrow \tau^{-n} P(i)$ or $\tau^{-n+1}P(i)\rightarrow \tau^{-n} P(j)$ for some natural number $n$. So the first case occurs when $T_i=\tau^{-p_j}P(i)$, hence $p_i=p_j$ and there is an irreducible map $T_j\rightarrow T_i$.
	The second case occurs when $T_i=\tau^{-p_j+1}P(i)$, so $p_i=p_j-1$ and there is an irreducible map $T_i\rightarrow T_j$, and these two cases are mutually exclusive by (a).  In both situations $[T_i]$ and $[T_j]$ are adjacent in the underlying graph of $\Gamma(\add T)$.
\end{claimproof}
We can now show that $\sigma$ is a graph automorphism of the underlying graph of $Q$.
By Claim \ref{claim3}, if $i\rightarrow j\in Q_1$, then $[T_i]$ and $[T_j]$ are adjacent in the underlying graph of $\Gamma(\add T)$. Conversely, irreducible morphisms can occur between two modules in $\tau$-orbits of $P(i)$ and $P(j)$ only when $i$ and $j$ are adjacent in the underlying graph of $Q$. Hence, the correspondence $i\mapsto T_i$ identifies the underlying graph of $\Gamma(\add T)$ with the underlying graph of $Q$. By Corollary \ref{closedundertau}, the irreducible morphisms in $\mathcal{F}(\St)$ (resp. $\mathcal{F}(\Cs)$) are the same those in $A\m$ between objects of $\mathcal{F}(\St)$ (resp. $\mathcal{F}(\Cs)$). By Lemma \ref{irreduciblefullsubcategories} and Lemma \ref{irreducibleprojectives} together with the fact that $\Psi$ sends irreducible morphisms of $\mathcal{F}(\St)$ to irreducible morphisms of $\mathcal{F}(\Cs)$ we obtain that $\Psi$ the irreducible morphisms $P(i)\rightarrow P(j)$ in $A\proj$ are sent to irreducible morphisms $T_{\sigma(i)}\rightarrow T_{\sigma(j)}$ in $\add T$. Thus, every edge $\{i, j\}$ of the underlying graph of $Q$ is sent to an edge $\{\sigma(i), \sigma(j)\}$ of the underlying graph of $Q$. So $\sigma$ is a graph automorphism of the underlying graph of $Q$ and (i) holds.

Let $i\rightarrow j\in Q_1$. Assume that $\sigma$ satisfies in addition that $\sigma^2=\id$ and $\sigma$ preserves the orientation of $i\rightarrow j$, that is, $\sigma(i)\rightarrow \sigma(j)\in Q_1$. So, there exists an irreducible map $P(\sigma(j))\rightarrow P(\sigma(i))$ and applying $\Psi$ we get an irreducible map $T_j\rightarrow T_i$. By Claim \ref{claim3}, we get $p_i=p_j$. Hence (ii) holds. 

Assume now that $\sigma(j)\rightarrow \sigma(i)\in Q_1$. Then there exists an irreducible map $P(\sigma(i))\rightarrow P(\sigma(j))$ and applying $\Psi$ we get an irreducible map $T_i\rightarrow T_j$. By Claim \ref{claim3}, $p_j=p_i+1$. So (iii) holds.

 It remains to show (iv). Denote by $l_i$ the number of indecomposable modules in the $\tau$-orbit of $P(i)$. Assume that $\sigma^2=\id$. Hence, $l_i=p_i+q_i+1=p_i+p_{\sigma(i)}+1 $.  If $i\rightarrow j\in Q_1$ and $\sigma(i)\rightarrow \sigma(j)\in Q_1$, then (ii) gives that $p_i=p_j$ and $p_{\sigma(i)}=p_{\sigma(j)}$ because $\sigma$ also preserves the orientation of $\sigma(i)\rightarrow \sigma(j)$. So $l_i=l_j$. Assume that $i\rightarrow j\in Q_1$ and $\sigma(j)\rightarrow \sigma(i)\in Q_1$. Then (iii) yields $p_j=p_i+1$ and $p_{\sigma(i)}=p_{\sigma(j)}+1$. Thus, $l_i=p_i+p_{\sigma(i)}+1=p_j-1+p_{\sigma(j)}+1+1=l_j$. Since $Q$ is connected we get that $l_i=l_j$ for all $i, j\in Q_0$.
\end{proof}

In practice, checking that a quasi-hereditary algebra is Ringel self-dual normally involves computing the Ringel dual and showing that it is Morita equivalent to the original algebra via an equivalence of categories that sends standard modules to standard modules. Presenting endomorphism algebras by quiver and relations can be laborious, especially for algebras with many simple modules. In the following we give a sufficient condition for a quasi-hereditary algebra to be Ringel self-dual which avoids computing the Ringel dual explicitly.

\begin{Theorem}\label{maintheoremoneimplication}
	Let $(A, \Lambda, \leq)$ be a (basic) quasi-hereditary algebra with $\gldim A\leq 1$. 
	\begin{enumerate}[(a)]
		\item  If $T=\tau^{-n}(A)$ for some $n\in \mathbb{N}$, then $\mathcal{F}(\St)=\add \oplus_{k=0}^n \tau^{-k} (A)= \add \oplus_{k=0}^n \tau^k (T)$. 
		\item If $T=\tau^{n}(DA)$ for some $n\in \mathbb{N}$, then $\mathcal{F}(\Cs)=\add \oplus_{k=0}^n \tau^{k} (DA)= \add \oplus_{k=0}^n \tau^{-k} (T)$. 
		\item If $T=\tau^{-n}(A)=\tau^n(DA)$ for some $n\in \mathbb{N}$, then  $(A, \Lambda, \leq)$ is Ringel self-dual and $A$ is of finite representation type.
	\end{enumerate}
\end{Theorem}
\begin{proof}
	Assume that  $T=\tau^{-n}(A)$ for some $n\in \mathbb{N}$.  Let $X\in A\m$ be a module satisfying $\Ext_A^1(X, T)=0$. 
	Then since $\gldim A\leq 1$ the Auslander-Reiten formula yields the following $k$-linear isomorphisms
	\begin{align*}
		0&=D\Ext_A^1(X, T)\cong \Hom_A(\tau^{-1}T, X)=\Hom_A(\tau^{-n-1}(A), X)\cong \Hom_A(\tau^{-n}(A), \tau X)\\&\cong \Hom_A(A, \tau^{n+1}(X))\cong \tau^{n+1}(X).
	\end{align*}
	Since $\tau^{n+1}(X)$ is zero if and only if $\tau^j X$ is projective for some $j\leq n$ (see for instance \citep[IV, 2.10]{zbMATH02228448}) it follows that $\mathcal{F}(\St)=\add \oplus_{k=0}^n \tau^{-k}(A)$. In particular, $\add \oplus_{k=0}^n \tau^{-k}(A)=\add \oplus_{k=0}^n \tau^k(T)$. So (a) holds.
	
		 (b) is the dual statement of (a), and so it is analogous. It also follows from (a) using the fact that $DT$ is the characteristic tilting module of $(A^{op}, \L, \leq)$ and $D\circ \tau^k=\tau^{-k} \circ D$.
	
	Assume now that $T=\tau^{-n}(A)=\tau^n(DA)$ for some $n\in \mathbb{N}$.
	Since $T$ is a tilting module it has $|A|$ non-isomorphic indecomposable direct summands and so $\tau^{-n}(P)\neq 0$ for all indecomposable projective modules and $\tau^n(I)\neq 0$ for all indecomposable injective modules. Thus, $\tau^{k}(T)$ has $|A|$ non-isomorphic indecomposable direct summands for every $k=-n, \ldots, n$. In particular, only the objects $\tau^n(T)$ are the projective modules.
So, the functor $\Psi\colon \add \oplus_{k=0}^n \tau^{-k} (T)\rightarrow \add \oplus_{k=0}^n \tau^k (T)$, induced by $X\mapsto \tau^n X$, is an equivalence of additive categories. By (a) and (b), $\Psi\colon \mathcal{F}(\Cs)\rightarrow \mathcal{F}(\St)$ is an equivalence of additive categories. It remains to show that $\Psi$ is an exact functor.
	
To do this, we need to show that 
\begin{align}
	\Hom_A(\tau^i X, A)=0, \quad \forall X\in \mathcal{F}(\Cs) \quad \forall i=0, \ldots, n-1. \label{eq1}
\end{align} Let $i\in \{0, \ldots, n-1\}$ and $k\in \{0, \ldots, n\}$. So $i-(n+k)$ is a negative integer. Then from the Auslander-Reiten formulas we obtain
$\Hom_A(\tau^{i-(n+k)}(A), A)\cong \Hom_A(\tau^{i-(n+k)+1}(A), \tau (A))=0$ (see for example \citep[IV, Corollaries 2.14, 2.15]{zbMATH02228448}). Thus
\[\Hom_A(\tau^i\oplus_{k=0}^n \tau^{-k}(T), A)\cong \Hom_A(\oplus_{k=0}^n \tau^{i-(n+k)}(A), A)=0.\]
Then, by (b), the equality \eqref{eq1} holds.

	Let $0\rightarrow X\rightarrow Y \rightarrow Z\rightarrow 0$ be an exact sequence on $\mathcal{F}(\Cs)$.
	We show by induction that the sequence $0\rightarrow \tau^k (X)\rightarrow \tau^k (Y)\rightarrow \tau^k(Z)\rightarrow 0$ is exact for every $k=0, \ldots, n$.
	For $k=0$, there is nothing to show. Assume that $0\rightarrow \tau^k (X)\rightarrow \tau^k (Y)\rightarrow \tau^k(Z)\rightarrow 0$ is exact for some $k<n$. 
	 Applying $\Hom_A(-, A)$ on it we get the exact sequence
	\begin{align*}
		\Hom_A(\tau^k(X), A)\rightarrow \Ext_A^1(\tau^k(Z), A)\rightarrow \Ext_A^1(\tau^k(Y), A)\rightarrow \Ext_A^1(\tau^k(X), A)\rightarrow 0.
	\end{align*}  By \eqref{eq1}, $\Hom_A(\tau^k(X), A)=0$.
Since $\tau=D\circ \Ext_A^1(-, A)$ we get that $0\rightarrow \tau^{k+1} (X)\rightarrow \tau^{k+1}(Y)\rightarrow \tau^{k+1}(Z)\rightarrow 0$ is an exact sequence. Therefore $\tau^n$ is exact on $\mathcal{F}(\Cs)$ and so $\Psi$ is an exact equivalence. Therefore $(A, \Lambda, \leq)$ is Ringel self-dual.
Since by (a) and (b) both $\mathcal{F}(\St)$ and $\mathcal{F}(\Cs)$ are of finite-type, then by Proposition \ref{prop} $A$ is of finite representation-type.
\end{proof}

\begin{proof}[Proof of Theorem \ref{maintheorem}]
	It follows by combining Lemma \ref{QuiveraddT} with Theorem \ref{maintheoremoneimplication}. The parity statement follows from Proposition \ref{counting}(c).
\end{proof}

\begin{Remark}
	The case $n=0$ in Theorem \ref{maintheoremoneimplication} is not treated because it is the semi-simple case. Indeed, if $T=A=DA$, then $\gldim A=\injdim A=0$. 
\end{Remark}

\section{Applications and classification}\label{Applications and classification}

In this section, we apply Theorem \ref{maintheorem} to classify the Dynkin diagrams for which there exists an orientation whose path algebra (over an algebraically closed field) admits a Ringel self-dual quasi-hereditary structure. We show that these are precisely the Dynkin diagrams of type $D_{2m}$ and $A_{4m+1}$. 
In particular, the smallest example of Ringel self-dual hereditary algebra of finite representation type is $kD_4$ where $D_4$ is equipped with the linear orientation for a certain ordering of the simples.

More precisely, the aim of this section is to prove the following classification.

\begin{Theorem} Let $\Gamma$ be Dynkin diagram of type $A$, $D$ or $E$. \label{maintheoremBproof}
	There exists a quiver $Q$ with underlying graph $\Gamma$ and a partial order $\leq$ on $Q_0$ making the corresponding quasi-hereditary algebra $(kQ, Q_0, \leq)$ Ringel self-dual if and only if $\Gamma$ is of type $D_{2m}$ or $A_{4t+1}$ for integers $m\geq 2$ and $t\geq 0$.
\end{Theorem}

\subsection{Examples of hereditary algebras that are not Ringel self-dual}\label{Examples of hereditary algebras that are not Ringel self-dual}

We start by considering the Dynkin diagrams for which Ringel self-duality is impossible. Type $D$ is the easiest case, since the required symmetry can already be ruled out by counting the total number of indecomposable
modules. Type $A$ is also settled by a counting argument but requires more refined information, namely the number of indecomposable modules in each $\tau$-orbit. Type $E$ requires further arguments, especially in types $E_7$ and $E_8$, where simple counting arguments are no longer sufficient. In these cases, our strategy is to show that the complete slice lying in the center of the Auslander-Reiten quiver cannot contain a simple module as direct summand.

\subsubsection{Type $D$}

We now show that path algebras of type $D_{2m+1}$ do not admit any Ringel self-dual quasi-hereditary structure, regardless of the orientation.

\begin{Cor}\label{corD4}
	If $n$ is odd, then $kD_n$ is not Ringel self-dual for any choice of orientation and partial ordering on the simples.
\end{Cor}
\begin{proof}
	The number of indecomposable $kD_n$-modules is $n(n-1)$. If it were Ringel self-dual, then by Proposition \ref{counting} we would have $n^2-n=n+2t$ for some natural number $t$, that is, $n^2$ would be even. But, by assumption, $n^2$ is odd.
\end{proof}

\subsubsection{Type $A$}

If we apply the same method to the Dynkin diagrams of type $A$ we get the following.

\begin{Cor}\label{CorAfirstcase}
	If $n\equiv 2 \mod 4$ or $n\equiv 3\mod 4$, then $kA_n$ is not Ringel self-dual for any choice of orientation and partial ordering on the simples.
\end{Cor}
\begin{proof}
	The number of indecomposable $kA_n$-modules is $\dfrac{n(n+1)}{2}$. Assume that $kA_n$ is Ringel self-dual, then by Proposition \ref{counting} we get that $n(n+1)=2n+4t$ for some natural number $t$. So $n(n-1)\equiv 0\mod 4$. But this condition only occurs for $n\equiv 0 \mod 4$ or $n-1\equiv 0 \mod 4$.
\end{proof}

Using this simple method we get already that the algebras $kA_2$ and $kA_3$ cannot be Ringel self-dual. 
Although this counting method is sufficient to exclude the majority of cases in type $D$, for type $A$ is not sufficient. Although the underlying graph of type $A$ can be settled by using that Ringel self-dual hereditary algebras of finite type have the same number of indecomposable modules in every $\tau$-orbit.

\begin{Cor}
	If $n \not\equiv 1 \mod 4$, then $kA_n$ is not Ringel self-dual for any choice of orientation and partial ordering on the simples. \label{excludingtypeA}
\end{Cor}
\begin{proof}
	Assume that there exists an orientation on $A_n$ and a partial ordering on $Q_0=\{1, \ldots, n\}$ that makes $(A, Q_0, \leq)$ Ringel self-dual.
	$A_n$ has two graph automorphisms $\sigma$ and both satisfy $\sigma^2=\id$.
		By Lemma \ref{QuiveraddT}, there exists a number $l$ such that every $\tau$-orbit of $A$ has $l$ indecomposable modules. Since $A$ has $n$ $\tau$-orbits we obtain
	\begin{align*}
		nl=\frac{n(n+1)}{2} \implies 2l=n+1.
	\end{align*}
Hence $n$ is odd. In particular, $n\not\equiv 0 \mod 4$. So the result follows from Corollary \ref{CorAfirstcase}.
\end{proof}

So far, we only discussed the underlying form of underlying Dynkin diagram of $A$, but of course the orientation imposed on the underlying Dynkin diagram also plays a role. 
 Nakayama hereditary algebras are Morita equivalent to algebras $kA_n$ where $A_n$ is linearly oriented. In the following we see that such algebras are never Ringel self-dual.
 
 \begin{Cor}
 	Let $A$ be a connected Nakayama hereditary algebra which is not semi-simple. Then there is no quasi-hereditary structure on $A$ making $A$ Ringel self-dual.
 \end{Cor}
\begin{proof}
	Assume that $A=kA_n$ for some natural number $n>1$ and that there exists a partial ordering $\leq$ on $Q_0=\{1, \ldots, n\}$ that makes $(A, Q_0, \leq)$ Ringel self-dual. Let $T$ be the basic characteristic tilting module. By Proposition \ref{counting}, $T$ is a complete slice module and the indecomposable direct summands of $T$ intersect all $\tau$-orbits. 
	$A$ has a unique projective-injective module  which is the unique indecomposable module with Loewy length $n$. Its $\tau$-orbit is a singleton and all simple modules belong to the same $\tau$-orbit (see for instance \citep[V, Corollary 4.2]{zbMATH02228448}).
	
	 Thus $T$ contains the unique projective-injective module and a simple module. By Lemma \ref{QuiveraddT}(a), the complete slice must be a directed line of length $n$ that contain the projective-injective module and a simple module as direct summands. 
	In the Auslander--Reiten quiver, the projective-injective module has exactly one outgoing arrow, to an injective module, and exactly one incoming arrow, from a projective module. Hence, the only directed lines of length $n$ satisfying these conditions are the line consisting of the $n$ indecomposable projective modules and the line consisting of the $n$ indecomposable injective modules.	
	 That is, $T=A$ or $T=DA$. In the first case, $\mathcal{F}(\Cs)=A\m$ while in the second case $\mathcal{F}(\St)=A\m$. By Proposition \ref{prop}(e), we obtain a contradiction so $A$ is not Ringel self-dual.
\end{proof}

\subsubsection{Type $E$}

In this part, we show that there is no Ringel self-duality in type $E$. Here, the case $E_6$ is significantly easier than $E_7$ and $E_8$ and with current tools we can already verify that Ringel self-duality does not occur for $E_6$.

\begin{Cor}\label{cor4dot5}
	The path algebra $kE_6$ is not Ringel self-dual for any choice of orientation and partial ordering on the simples.
\end{Cor}
\begin{proof}
	We consider $E_6$ with the following labelling:
	\begin{equation*}
		\begin{tikzcd}[row sep=small, column sep=small]
			1 \arrow[r, dash] & 2 \arrow[r, dash] & 3 \arrow[r, dash] \arrow[d, dash] & 4 \arrow[r, dash] & 5 \\
			& & 6 & &
		\end{tikzcd}.
	\end{equation*}
	$E_6$ admits two graph automorphisms: the identity and $\sigma\in S_6$ where $\sigma=(1 5)(2 4)$. Let $l_6$ be the number of indecomposable modules in the $\tau$-orbit of $P(6)$.
	By Lemma \ref{sizetauorbitsADE}, $2l_6=12$. If $kE_6$ would be Ringel self-dual then by Lemma \ref{QuiveraddT}  we would obtain $6=l_6=p_6+q_6+1=2p_6+1$ for some $p_6\in \mathbb{N}$. So, the claim follows. 
\end{proof}

For the cases $E_7$ and $E_8$, Lemma \ref{QuiveraddT} with Lemma \ref{sizetauorbitsADE} do not settle the question on their own, but they restrict the possible characteristic tilting modules. Indeed, since the Dynkin diagrams $E_7$ and $E_8$ have trivial automorphism group, combining the two lemmas shows: if $(kQ, Q_0, \leq)$ is Ringel self-dual for a quiver $Q$ of type $E_7$ or $E_8$ and a partial order on $Q_0$, then the indecomposable direct summands of the corresponding characteristic tilting module form a complete slice lying centrally in the Auslander-Reiten quiver of $kQ$. Our approach is to show that no such central slice contains a simple module, contradicting the fact that $T(i)=S(i)$ for $i$ minimal with respect to $\leq$.

\begin{Lemma}\label{reductionfortypeE}
	Let $Q$ be a quiver of type $E_7$ or $E_8$ with an arbitrary orientation, and let $j$
	be a vertex of $Q$ such that $j\rightarrow b_1$ is the unique arrow starting at $j$.
	Denote by $\mathcal{B}$ the full subquiver of $Q$ supported on the connected component
	of $b_1$ in $Q\setminus\{j\}$, and assume that $\mathcal{B}$ is a path with $b_1$ as an
	endpoint, say
	\[
	\mathcal{B}\colon \begin{tikzcd}[row sep=small, column sep=small]
		b_1 \arrow[r, dash] & b_2 \arrow[r, dash] & \cdots  \arrow[r, dash] & b_l
	\end{tikzcd}
	\qquad (l\leq 4),
	\]
	the vertices being numbered consecutively starting at $b_1$. Put
	$m_Q(j):=\min\{m\geq 0\colon \tau^{m}S(j)\ \text{is projective}\}$. Then
	\[
	m_Q(j)=1+|\{1\leq s\leq l-1\colon b_s\rightarrow b_{s+1}\in Q_1\}|.
	\]
\end{Lemma}
\begin{proof}
	Given $M\in kQ\m$, we write $m_Q(M):=\min\{m\geq 0\colon \tau^{m}M\ \text{is projective}\}$.
	Since $j$ is not a sink, then $S(j)$ is not projective. We shall compute the $\tau$-orbit of $S(j)$ in $kQ$. We can observe that $\rad P(j)=P(b_1)$ so $0\rightarrow P(b_1)\rightarrow P(j)\rightarrow S(j)\rightarrow 0$ is the minimal projective resolution of $S(j)$. Hence $\tau S(j)=\ker (I(b_1)\rightarrow I(j))$. The source of the arrows that reach $b_1$ are either from the path $\mathcal{B}$ or from $j$. Moreover, since $kQ$ is hereditary, the map $I(b_1)\rightarrow I(j)$ is surjective.
	 So there exists an exact sequence $0\rightarrow I_\mathcal{B}(b_1)\rightarrow I(b_1)\rightarrow I(j)\rightarrow 0$. Given $\mathcal{B}$ a subquiver of $Q$ we write $I_\mathcal{B}(b_1)$ to denote injective envelope of $S(b_1)$ in $k\mathcal{B}$ (where $k\mathcal{B}$ is a quotient of $kQ$). Thus, $\tau S(j)=I_\mathcal{B}(b_1)$ and so
	 \begin{align}
	 	m_Q(j)=1+m_Q(I_\mathcal{B}(b_1)) \label{eq4}
	 \end{align}
	We write $M_{i, k}$ for the module $\begin{matrix}
		b_k \\ \vdots \\ b_{i+1} \\ b_i
	\end{matrix}$ for $i\in \{1, 2, 3, 4\}$ and $k\in \{i+1, \ldots, 4\}$.
	The form of $I_\mathcal{B}(b_1)$ depends on the orientation on $\mathcal{B}$ and on $l$.  In total, there are four possible forms for $I_\mathcal{B}(b_1)$:
	\begin{enumerate}
		\item $I_\mathcal{B}(b_1)=S(b_1)$ if no arrows ends at $b_1$ in $\mathcal{B}$;
		\item $I_\mathcal{B}(b_1)=M_{1, 2}$ if there exists an arrow $b_2\rightarrow b_1$ and no arrow ends at $b_2$;
	\item $I_\mathcal{B}(b_1)=M_{1, 3}$ if $b_3\rightarrow b_2\rightarrow b_1$ is a subquiver of $\mathcal{B}$ and no arrow ends at $b_3$;
\item $I_\mathcal{B}(b_1)=M_{1, 4}=P(b_4)$ if $\mathcal{B}$ is the quiver $b_4\rightarrow b_3\rightarrow b_2\rightarrow b_1$.
	\end{enumerate}
To determine $m_Q(I_\mathcal{B}(b_1))$ we
prove the following statement by induction on $r=|\mathcal{P}|$: for every arrow
$x\rightarrow x_1$ of $Q$ such that the component
$	\mathcal{P}\colon \begin{tikzcd}[row sep=small, column sep=small]
	x_1 \arrow[r, dash] & x_2 \arrow[r, dash] & \cdots  \arrow[r, dash] & x_r
\end{tikzcd}$ of
$x_1$ in $Q\setminus\{x\}$ is a path with endpoint $x_1$, one has
\[
m(I_{\mathcal{P}}(x_1))=|\{1\leq s\leq r-1\colon x_s\rightarrow x_{s+1}\in Q_1\}|.\]

If $r=1$, then $\mathcal{P}$ is just the vertex $x_1$ and so $I_\mathcal{P}(x_1)=S(x_1)=P(x_1)$ thus in such a case $m(I_\mathcal{P}(x_1))=0$.
Assume that $r>1$. If there is an arrow $x_1\rightarrow x_2$, then $I_\mathcal{P}(x_1)=S(x_1)$ and $0\rightarrow P(x_2)\rightarrow P(x_1)\rightarrow  S(x_1)\rightarrow 0$ is the minimal projective resolution of $S(x_1)$. Thus $\tau S(x_1)=\ker (I(x_2)\rightarrow I(x_1))=I_{\mathcal{P}'}(x_2)\neq 0$, where $	\mathcal{P}'\colon \begin{tikzcd}[row sep=small, column sep=small]
	x_2 \arrow[r, dash] & \cdots  \arrow[r, dash] & x_r
\end{tikzcd}$. So by induction \begin{align*}
	m(I_\mathcal{P}(x_1))=1+m(I_{\mathcal{P}'}(x_2))=1+|\{2\leq s\leq r-1\colon x_s\rightarrow x_{s+1}\in Q_1\}|=|\{1\leq s\leq r-1\colon x_s\rightarrow x_{s+1}\in Q_1\}|.
\end{align*}
Assume now that there exists an arrow $x_2\rightarrow x_1$. Let $k$ be maximal with $x_k\rightarrow \cdots \rightarrow x_2\rightarrow x_1$. Then $I_\mathcal{P}(x_1)=M_{1, k}$ and $x_k$ is a source, thus $I(x_k)=S(x_k)$. If $k=r$, then $I_\mathcal{P}(x_1)=P(x_r)$ and so $m(I_\mathcal{P}(x_1))=0$. If $k<r$, then $0\rightarrow P(x_{k+1})\rightarrow P(x_k)\rightarrow M_{1, k}\rightarrow 0$ is exact and so $\tau M_{1, k}=\ker (I(x_{k+1})\rightarrow I(x_k) )=I_{\mathcal{P}'}(x_{k+1})\neq 0$, with $	\mathcal{P}'\colon \begin{tikzcd}[row sep=small, column sep=small]
	x_{k+1} \arrow[r, dash] & \cdots  \arrow[r, dash] & x_l
\end{tikzcd}$. By induction, 
\begin{align*}
	m(I_\mathcal{P}(x_1))=1+m(I_{\mathcal{P}'}(x_{k+1}))=1+|\{k+1\leq s\leq r-1\colon x_s\rightarrow x_{s+1}\in Q_1\}|=|\{1\leq s\leq r-1\colon x_s\rightarrow x_{s+1}\in Q_1\}|.
\end{align*}
Thus, from $m_Q(I_\mathcal{B}(b_1))=|\{1\leq s\leq l-1\colon b_s\rightarrow b_{s+1}\in Q_1\}|$ and by \eqref{eq4} the claim follows.
\end{proof}

\begin{Lemma}\label{reductiontoOp}
	Let $Q$ be a quiver with $n$ vertices and $kQ$ the path algebra of $Q$. Put
	$m_Q(j):=\min\{m\geq 0\colon \tau^{m}S(j)\ \text{is projective}\}$. Assume that all $\tau$-orbits admit $l+1$ indecomposable modules. Then $m_{Q^{op}}(j)=l-m_Q(j)$.
\end{Lemma}
\begin{proof}
	We write $S_Q(j)$ to denote the simple module over $kQ$ indexed by $j$ while $S_{Q^{op}}(j)$ denotes the simple modules over $kQ^{op}$ indexed by $j$. Set $A=kQ$.
	Observe that $D\tau_A=DD\Tr=\Tr DD=\tau^{-1}_{A^{op}}D$. 
	
	 By assumption, $\tau_{A^{op}}^{-m_{Q^{op}}(j)} P_{A^{op}}(i)\cong S_{Q^{op}}(j)$ for some $i$. Hence \begin{align}
	 	S_Q(j)\cong DS_{Q^{op}}(j)\cong D\tau_{A^{op}}^{-m_{Q^{op}}(j)} P_{A^{op}}(i)\cong \tau_A^{m_{Q^{op}}(j)} D P_{A^{op}}(i)\cong \tau_A^{m_{Q^{op}}(j)}  I_A(i)
	 \end{align}
 Therefore, there exists some $k$ such that
 \begin{align}
 	S_Q(j)\cong \tau_A^{m_{Q^{op}}(j)}  I_A(i) \cong \tau_A^{m_{Q^{op}}(j)}  \tau^{-l}P_A(k)\cong \tau_A^{-(l-m_{Q^{op}}(j))} P_A(k).
 \end{align}
Thus, the claim follows.
\end{proof}

\begin{Theorem}\label{tauorbitsE7E8}
	\begin{enumerate}
		\item Let $Q$ be a quiver of type $E_7$ with an arbitrary orientation. The following assertions hold.
		\begin{enumerate}
			\item Then every $\tau$-orbit of the path algebra $kQ$ has 9 indecomposable modules.
			\item For every vertex $i\in Q_0$ the module $\tau^{-4}P(i)$ is non-zero and not simple.
		\end{enumerate} 
		\item Let $Q$ be a quiver of type $E_8$ with an arbitrary orientation. The following assertions hold.
		\begin{enumerate}
			\item every $\tau$-orbit of the path algebra $kQ$ has 15 indecomposable modules.
			\item For every vertex $i\in Q_0$, the module $\tau^{-7}P(i)$ is non-zero and not simple.
		\end{enumerate}
	\end{enumerate}
\end{Theorem}
\begin{proof}
The Dynkin diagrams $E_7$ and $E_8$ have trivial automorphism groups (see for instance \cite{zbMATH01701639}), thus parts 1(a) and 2(a) follow from Lemma \ref{sizetauorbitsADE}. For a given vertex $j$ of $Q$, write $m_Q(j):=\min\{m\geq 0\colon \tau^{m}S(j)\ \text{is projective}\}$.

	Assume first that $Q$ is of type $E_7$, where $E_7$ is considered with the following labelling:
\begin{equation*}
	\begin{tikzcd}[row sep=small, column sep=small]
		1 \arrow[r, dash] & 2 \arrow[r, dash] & 3 \arrow[r, dash] \arrow[d, dash] & 4 \arrow[r, dash] & 5 \arrow[r, dash] & 6 \\
		& & 7 & &
	\end{tikzcd}.
\end{equation*} 
Let $j$ be a vertex of $Q$. To prove 1(a) it is enough to show that $m_Q(j)$ is never $4$. If $j$ is a sink, then $P(j)=S(j)$, so $m_Q(j)=0$. If $j$ is a source, then $I(j)=S(j)$ and since every $\tau$-orbit has $9$ indecomposable modules we get that $m_Q(j)=8$. Let $d$ be the distance in the underlying graph $E_7$. Assume that $j$ has degree two which is not a source nor a sink. So there exists a unique vertex $b_1$ such that $d(b_1, 3)=d(j, 3)+1$ and either $j\rightarrow b_1\in Q_1$ or $b_1\rightarrow j\in Q_1$. Assume that $j\rightarrow b_1\in Q_1$. So the full subquiver of $Q$ supported on the connected component of $b_1$ in $Q\setminus \{j\}$ has at most two vertices. 
By Lemma \ref{reductionfortypeE}, $m_Q(j)\leq 1+1=2$. If $b_1\rightarrow j\in Q_1$, then $j\rightarrow b_1\in (Q^{op})_1$ and by Lemma \ref{reductionfortypeE}, it follows that $m_{Q^{op}}(j)\leq 2$. By 
Lemma \ref{reductiontoOp}, we get $m_Q(j)\geq 8-2=6$. So again, $m_Q(j)\neq 4$.
It remains to consider the case $j=3$. If $j$ is a source or a sink, then we have seen already that $m_Q(j)\neq 4$. Suppose therefore that $j$ is neither a source nor a sink. Since $j$ is trivalent, either exactly one arrow starts at $j$ and two arrows end at $j$, or the other way around. These two cases are exchanged by passing to $Q^{op}$, and $m_Q(j)\neq 4$ if and only if $m_{Q^{op}}(j)\neq 4$ by Lemma \ref{reductiontoOp}  and part 1(a). So we may assume without loss of generality that exactly one arrow starts at $j$ and two arrows end at $j$. 
By Lemma \ref{reductionfortypeE}, we obtain that $m_Q(j)\neq 4$. So the claim 1(b) holds.

Assume now that $Q$ is of type $E_8$, where $E_8$ is considered with the following labelling:
\begin{equation*}
	\begin{tikzcd}[row sep=small, column sep=small]
		1 \arrow[r, dash] & 2 \arrow[r, dash] & 3 \arrow[r, dash] \arrow[d, dash] & 4 \arrow[r, dash] & 5 \arrow[r, dash] & 6 \arrow[r, dash]  & 7\\
		& & 8 & &
	\end{tikzcd}.
\end{equation*} 
To show 2(b), we need to show that $m_Q(j)\neq 7$ for every vertex $j$.
By part 2(a) and Lemma \ref{reductiontoOp}, $m_Q(j)=14-m_{Q^{op}}(j)$ for every vertex $j$ of $Q$. So, $m_Q(j)\neq 7$ if and only $m_{Q^{op}}(j)\neq 7$. Let $j\in Q_0$. Again if $j$ is a source or a sink, then $m_Q(j)\in \{0, 14\}$. So, we get that $m_Q(j)\neq 7$ for $j\in \{1, 7, 8\}$. For $j\in \{2, 4, 5, 6\}$ using the same strategy as for $E_7$ we get that if $j$ is not a source nor a sink, then $m_Q(j)\leq 3$ or $m_Q(j)\geq 14-3=11$. For $j=3$, we can again assume without loss of generality that exactly one arrow starts at $j$, say $j\rightarrow b_1$ and two arrows end at $j$. The possible full subquivers of $Q$ supported on the connected component of $b_1$ in $Q\setminus \{3\}$ are the subquivers with vertices $\{1, 2\}$ or $\{8\}$ or $\{4, 5, 6, 7\}$. By Lemma \ref{reductionfortypeE}, we obtain that $m_Q(j)\leq 4$. So the claim 2(b) holds.
\end{proof}

\begin{Cor}\label{RsdfortypeE78}
	Let $Q$ be a quiver of type $E_7$ or $E_8$ with an arbitrary orientation. The path algebra $kQ$ is not Ringel self-dual for any partial ordering of the simples.
\end{Cor}
\begin{proof}
	Since $E_7$ and $E_8$ have trivial automorphism groups, we obtain that by Lemma \ref{QuiveraddT} and Theorem \ref{tauorbitsE7E8} that if $(kQ, Q_0, \leq)$ is Ringel self-dual for some partial order $\leq$, then the characteristic tilting module satisfies $T=\tau^{-4}(kQ)$ or $T=\tau^{-7}(kQ)$, the first occuring for $E_7$ and the second for $E_8$. Any characteristic tilting module contains at least one simple direct summand, so by Theorem \ref{tauorbitsE7E8} the result follows.
\end{proof}

\begin{Remark}
	Observe that in Corollary \ref{cor4dot5} we used Lemma \ref{sizetauorbitsADE}, but the same result could be deduced directly using Lemma \ref{QuiveraddT}. Similarly, to deduce that hereditary algebras with underlying Dynkin diagram of the form $E_7$ or $E_8$ are not Ringel self-dual we could have only used Lemma \ref{QuiveraddT}.
\end{Remark}

\subsection{Examples of Ringel self-dual hereditary algebras and classification}

The goal of this subsection is to illustrate the existence of orientations in the Dynkin diagrams $D_{2m}$ and $A_{4m+1}$.  

\subsubsection{Type A}

We start by showing that Ringel self-duality can happen in type $A$. 

\begin{Prop}\label{Proptauorbitsprediction}
	Let $n=4m+1$ for some $m\in \mathbb{N}$ and fix an orientation on $A_n$ and a partial order $\leq$ on $Q_0=\{1, \ldots, n\}$ so that $(kA_n, Q_0, \leq)$ is Ringel self-dual. Then every $\tau$-orbit contains exactly $2m+1$ indecomposable modules and the characteristic tilting module of $kA_n$ is $\tau^{-m}(kA_n)$.
\end{Prop}
\begin{proof} Let $T$ be the characteristic tilting module and denote by $Q$ the underlying quiver of $kA_n$. For every $i\in Q_0$ there are numbers $p_i, q_i\in \mathbb{N}_0$ such that $\tau^{-p_i}P(i)\in \add T$ and $\tau^{-p_i-q_i}P(i)$ is a non-zero injective module. We claim that $p_i=q_i$ for every $i\in Q_0$.
	$A_n$ admits two graph automorphisms, namely the identity and the reflection $\rho$ given by $\rho(i)=n-i+1$. By Lemma \ref{QuiveraddT}, if both of these two graph automorphisms preserve the given orientation of $A_n$ then since $A_n$ is connected we would obtain that $p_i=p_j$ for all $i, j\in Q_0$ and consequently $p_i=q_i$ for all $i\in Q_0$. The identity clearly preserves the given orientation of $A_n$. By Lemma \ref{QuiveraddT}(b)(iv),
	every $\tau$-orbit contains exactly $l$ indecomposables, for some $l$.
	Since $A$ has $\frac{n(n+1)}{2}$ indecomposable modules we have that $nl=\frac{n(n+1)}{2}$ and so $l=\frac{n+1}{2}=2m+1$ and $2m+1=l=p_i+q_i+1$ for all $i\in Q_0$.
	
	Assume that $i\rightarrow j\in Q_1$. So there exists an irreducible map $P(j)\rightarrow P(i)$ and so an irreducible map $\tau^{-2m}P(j)\rightarrow \tau^{-2m}P(i)$. Since all $\tau$-orbits have the same length, Lemma \ref{sizetauorbitsADE} yields that this map corresponds to an irreducible map $I(\rho(j))\rightarrow I(\rho(i))$. Since $\Gamma(A\inj)=Q^{op}$ this irreducible map arises from an arrow $\rho(i)\rightarrow \rho(j)\in Q_1$. Thus, $\rho$ preserves the orientation of $Q$ and so it follows by the discussion above that $p_i=q_i=m$ for all $i\in Q_0$. Therefore, $T=\tau^{-m}(kA_n)=\tau^m(D(kA_n))$.
\end{proof}

Let $m\geq 1$ and let $Q$ be the quiver 
\begin{equation}\label{eq:D2m-quiver}
	\begin{tikzcd}[row sep=small, column sep=small]
		1 \ar[r] & \cdots \ar[r] & m+1 \ar[r] & \cdots \ar[r] & 2m+1  & \ar[l] \cdots \ar[l] & 3m+1 \ar[l] & \cdots \ar[l] & \ar[l] 4m+1  \\
	\end{tikzcd}
\end{equation} and let $R$ the path algebra $kQ$. We write $S(i),P(i),I(i)$ for the simple, indecomposable projective and
indecomposable injective $R$-module at vertex $i$, and $\varepsilon_1,\dots,\varepsilon_{4m+1}$ for the
standard basis of $K_0(R)\cong\mathbb Z^{4m+1}$. The indecomposable projective $R$-modules are the uniserial modules
\begin{align*}
	P(i)=\Loe{i \\ i+1 \\ \vdots \\ 2m+1 }\quad (1\le i\le 2m), \quad P(2m+1)=\Loe{2m+1}, \quad P(i)=\Loe{i \\ i-1 \\ \vdots \\ 2m+1 }\quad (2m+2\le i\le 4m+1)
\end{align*}
while the indecomposable injective modules of $R$ are the modules
\begin{align*}
	\quad I(i)=\Loe{1 \\ \vdots \\ i} \quad (1\leq i\leq 2m), \quad 	I(2m+1)=\Loe{1 & & 4m+1 \\ 2 & & 4m \\ \vdots & & \vdots \\ 2m & & 2m+2 \\ & 2m+1 & },\quad I(i)=\Loe{4m+1 \\ \vdots \\ i} \quad (2m+2\le i\le 4m+1).
\end{align*}
In particular, $I(1)=S(1)$ and $I(4m+1)=S(4m+1)$.
We equip $Q_0=\{1,\dots,4m+1\}$ with the partial order $\preceq$ generated by the relations
\begin{align*}i \prec j \quad \text{whenever} \quad  \begin{cases}
		i>j \text{ and } i, j\in \{1, \ldots, m\}\cup \{2m+1, \ldots, 3m+1\} \\
		i<j \text{ and } i, j\in \{m+1, \ldots, 2m+1\}\cup \{3m+2, \ldots, 4m+1\}
	\end{cases}.
\end{align*}
In particular for $m=1$, $\preceq$ is generated by the relations $2\prec 3\prec 5$ and $4\prec 3\prec 1$.

We can see that $(R, Q_0, \preceq)$ is a quasi-hereditary algebra whose standard modules are
\[\St(i)=\begin{cases}
	S(i), \quad m+1\leq i\leq 3m+1\\
	P(i), \quad 1\leq i\leq m  \text{ or } i\geq 3m+2
\end{cases}. \]
The respective costandard modules are the indecomposable modules with dimension vectors given by
\begin{align*} 
	\dv \Cs(i)=\begin{cases}
		\displaystyle\sum_{j=m+1}^{i} \varepsilon_j	& m+2\leq i \leq 2m\\[1.5em]
		\displaystyle\sum_{j=m+1}^{3m+1} \varepsilon_j	& i=2m+1\\[1.5em]
		\displaystyle\sum_{j=i}^{3m+1} \varepsilon_j	& 2m+2\leq i\leq 3m \\[1.5em]
		\varepsilon_i, & 1\leq i\leq m+1 \text{ or } 3m+1\leq i\leq 4m+1
	\end{cases}.
\end{align*}
In particular, $\Cs(i)$ is simple if and only if $i\notin \{m+2, \ldots, 3m\}.$ 
It follows that $\Cs(i)\in \mathcal{F}(\St)$ for $m+1\leq i\leq 3m+1$ and thus $T(i)=\Cs(i)$ for $m+1\leq i\leq 3m+1$.

The remaining indecomposable direct summands of the characteristic tilting module $T$ of $(R, Q_0, \preceq)$ are the indecomposable modules with dimension vectors given by
\begin{align*}
	\dv T(i)=\begin{cases}
		\displaystyle	\sum_{j=i}^{3m+1} \varepsilon_j & 1\leq i\leq m\\[1.5em]
		\displaystyle	\sum_{j=m+1}^i \varepsilon_j & 3m+2\leq i \leq  4m+1
	\end{cases}.
\end{align*}

We want to show that $(R, Q_0, \preceq)$ is Ringel self-dual. Our strategy is the same as in type $D$: first we confirm that every $\tau$-orbit has $2m+1$ indecomposable modules as Proposition \ref{Proptauorbitsprediction} predicts and then to show that the characteristic tilting module of $R$ is "central" in the Auslander-Reiten quiver of $R$.
To prove both of these statements we will make use of the Coxeter matrix of $R$.

Set $c:=2m+1$ and $N:=4m+1$.
\begin{Lemma}\label{coxetertypeA}
	Let $C$ be the Cartan matrix of $R$. The Coxeter matrix $\Phi=-C^TC^{-1}$ of $R$ in terms of the standard basis $\varepsilon_1, \ldots, \varepsilon_{4m+1}$ is given by
	\[\Phi \varepsilon_k=\varepsilon_{k+1} \ (1\leq k\leq c-2), \quad \Phi\varepsilon_{c-1}=\sum_{j=c}^{N} \varepsilon_j, \quad \Phi\varepsilon_c = -\sum_{j=1}^{N} \varepsilon_j, \quad \Phi\varepsilon_{c+1}= \sum_{j=1}^{c} \varepsilon_j, \quad \Phi\varepsilon_k = \varepsilon_{k-1} \ (c+2\leq k\leq N)\]
	while the inverse of the Coxeter matrix is given by
	\begin{align*}
		\Phi^{-1}\varepsilon_1=-\sum_{j=1}^{c}\varepsilon_j, \quad \Phi^{-1}\varepsilon_k&=\varepsilon_{k-1} \ (2\leq k\leq c-1), \quad \Phi^{-1}\varepsilon_c=\varepsilon_{c-1}+\varepsilon_c+\varepsilon_{c+1}, \\ \Phi^{-1}\varepsilon_k&=\varepsilon_{k+1} \ (c+1\leq k\leq N-1), \quad \Phi^{-1}\varepsilon_N=-\sum_{j=c}^{N}\varepsilon_j.
	\end{align*}
\end{Lemma}
\begin{proof}
	By reading off the Loewy series of the $P(j)$, $(i, j)$-entry of the Cartan matrix $C$ is given by
	\[
	C_{ij}=[P(j):S(i)]= \begin{cases}
		1, & j\leq i\leq c \ \text{ or }\ c\leq i\leq j,\\
		0, & \text{otherwise,}
	\end{cases}
	\]
	In particular, the $j$-th column of $C$ is $\dv P(j)$ and the $j$-th column of $C^T$ is $\dv I(j)$. It is clear that $C^{-1}\varepsilon_c=\varepsilon_c$. Assume that $j\geq c$. By induction with base case $j=c$ we get from $\varepsilon_{j+1}=C^{-1}C\varepsilon_{j+1}=C^{-1}(\varepsilon_{j+1})+\sum_{k=c}^j C^{-1}\varepsilon_k$ that the ${j+1}$-th column of $C^{-1}$ is $C^{-1}\varepsilon_{j+1}=\varepsilon_{j+1}-\varepsilon_j$. Similarly, the $j$-th column of $C^{-1}$ is $\varepsilon_j-\varepsilon_{j+1}$ whenever $j<c$. 
	Thus \begin{align}
		\Phi \varepsilon_j=\begin{cases}
			-C^T(\varepsilon_j-\varepsilon_{j+1}) & j<c \\
			-C^T \varepsilon_c & j=c \\
			-C^T(\varepsilon_j-\varepsilon_{j-1}) & j>c
		\end{cases}\ =\ \begin{cases}
			\dv I(j+1)-\dv I(j),  & j<c\\
			\dv I(c) & j=c\\
			\dv I(j-1)-\dv I(j) & j>c
		\end{cases}. \label{eqinjectivescoxeter}
	\end{align} So the statement about the Coxeter matrix follows by reading off the Loewey series of the injectives together with \eqref{eqinjectivescoxeter}.
	The statement about the inverse of the Coxeter matrix is
	checked directly by applying $\Phi$ to the right-hand sides.
\end{proof}

Given $M\in R\m$, write $m(M):=\min \{r\geq 0\colon \tau^r M\in R\proj\}$.

\begin{Lemma} The following assertions hold for the algebra $R$. \label{tauorbitsofsimplestypeA}
	\begin{enumerate}[(a)]
		\item Let $i\in Q_0$. Then $m(S(i))=\begin{cases}
			c-i, & 1\leq i\leq c-1\\
			i-c & c\leq i\leq N	
		\end{cases}.$ In particular, $m(S(i))\leq c-1$.
		\item Every $\tau$-orbit contains exactly $c=2m+1$ indecomposable $R$-modules. 
	\end{enumerate}
\end{Lemma}
\begin{proof}
	The module $S(c)$ is projective so $m(S(c))=0$.  By Lemma \ref{coxetertypeA}, $\Phi \varepsilon_{c-1}=\varepsilon_c+\cdots+\varepsilon_N$ and thus $\dv \tau S(c-1)=\dv P(N)$, that is, $m(S(c-1))=1$.  Also $\Phi \varepsilon_{c+1}=\sum_{j=1}^{c} \varepsilon_j=\dv P(1)$, that is $m(S(c+1))=1$.
	By Lemma \ref{coxetertypeA},
	$\dv \tau S(k)=\dv S(k+1)$ for $1\leq k\leq c-2$ and $\dv \tau S(k)=S(k-1)$ for $c+2\leq k\leq N$. Hence \[m(S(k))=\begin{cases}
		m(S(k+1))+1, & 1\leq k\leq c-2 \\
		m(S(k-1))+1, & c+2\leq k\leq N
	\end{cases}.\] By induction, we get that $m(S(k))=c-k$ for $1\leq k\leq c-1$ and $m(S(k))=k-c$ for $c+1\leq k\leq N$. So (a) holds.
	
	By part (a), $\tau^c S=0$ for every simple module $S$. By Lemma \ref{tauvanishingonsimples}, $\tau^c=0$. So each $\tau$-orbit contains at most $c$ indecomposable modules, namely $I(i), \tau I(i), \ldots, \tau^{c-1} I(i)$ for $i\in Q_0$. Since $R$ contains exactly $(4m+1)(2m+1)=cN$ indecomposable modules, by setting $o_i$ the number of indecomposable modules in the $\tau$-orbit of $I(i)$ we get the following inequality
	\[cN = \sum_{i=1}^{N} o_i\leq \sum_{i=1}^{N} c = cN.\] Since $o_i\geq 1$ we get that $o_i=c$ for all $i=1, \ldots, N$. So (b) holds.
\end{proof}

\begin{Theorem}\label{thm4dot20}
	The quasi-hereditary algebra $(R, Q_0, \preceq)$ is Ringel self-dual.
\end{Theorem}
\begin{proof} 
	We want to show that $T=\tau^{-m}(R)$. Observe that $T(i)\twoheadrightarrow S(3m+1)$ for all $1\leq i\leq m$ and $2m+1\leq i\leq 3m+1$ and $T(i)\twoheadrightarrow S(m+1)$ for all $i\in \{m+1, \ldots, 2m\}\cup \{3m+2, \ldots, 4m+1\}$.
	By Lemma \ref{tauorbitsofsimplestypeA}, 
	$m(S(3m+1))=m=m(S(m+1))$. Thus it follows by Lemma \ref{taumeasure} that $m(T(i))\leq m$ for all $i\in Q_0$. Observe also that $\Hom_R(T(2m+1), T(i))\neq 0$ for all $i\in Q_0$, thus
	\[m(T(2m+1))\leq m(T(i))\leq m \quad \forall i\in Q_0.\]
	So we need to compute $m(T(c))=m(\Cs(c))$, where $c=2m+1$. 
	
	We proceed by induction on $i$ to show that $\Phi^{-i}(\varepsilon_c)=\displaystyle \sum_{k=c-i}^{k=c+i} \varepsilon_k$ for $0\leq i\leq m$. The case $i=0$ is clear, while $i=1$ follows from Lemma \ref{coxetertypeA}. For $m\geq i>0$, induction and Lemma \ref{coxetertypeA} yields that
	\begin{align*}
		\Phi^{-(i+1)}(\varepsilon_c)&=\Phi^{-1}\left(\sum_{k=c-i}^{c+i}\varepsilon_k \right) =\Phi^{-1}\varepsilon_c+\sum_{k=c-i}^{c-1}\Phi^{-1}\varepsilon_k+ \sum_{k=c+1}^{c+i}\Phi^{-1}\varepsilon_k\\&=\varepsilon_{c-1}+\varepsilon_c+\varepsilon_{c+1}+\sum_{k=c-i}^{c-1}\varepsilon_{k-1}+\sum_{k=c+1}^{c+i}\varepsilon_{k+1}=\sum_{k=c-i-1}^{c+i+1}\varepsilon_k.
	\end{align*}Since $P(c)=S(c)$ we obtain that
	$\dv \tau^{-m}P(c)=\displaystyle \sum_{k=m+1}^{k=3m+1} \varepsilon_k=\dv \Cs(c)$, and consequently $m(\Cs(c))=m$. Thus, we get that $T=\tau^{-m}(R)$. By Lemma \ref{tauorbitsofsimplestypeA}(b), it follows that $T=\tau^m(DR)$. By Theorem \ref{maintheoremoneimplication}, $(R, Q_0, \preceq)$ is Ringel self-dual.
\end{proof}

\begin{Remark}
	The algebra $(R, Q_0, \preceq)$ is Ringel self-dual, but $R$ is not isomorphic to its opposite algebra $R^{op}$.
\end{Remark}

\subsubsection{Type D}

In this subsubsection, for $m\geq 2$ we show that the path algebra of the linearly oriented Dynkin diagram $D_{2m}$   admits a Ringel self-dual quasi-hereditary structure.

Let $m\geq 2$ and let $Q$ be the quiver 
\begin{equation}\label{eq:D2m-quiver}
	\begin{tikzcd}[row sep=small, column sep=small]
		& & & & & & 1 \\
		2m \ar[r] & 2m-1 \ar[r] & \cdots \ar[r] & m \ar[r] & \cdots \ar[r] & 3 \ar[ru]\ar[rd] & \\
		& & & & & & 2
	\end{tikzcd}
\end{equation} and let $R$ be the path algebra $kQ$. We write $S(i),P(i),I(i)$ for the simple, indecomposable projective and
indecomposable injective $R$-module at vertex $i$, and $\varepsilon_1,\dots,\varepsilon_{2m}$ for the
standard basis of $K_0(R)\cong\mathbb Z^{2m}$, so that $\dv M=(\dim M_i)_i$ is a
column vector for $M\in R\m$. As usual, we display the modules  by their Loewey series. The indecomposable projective modules of $R$ are the modules
\begin{align*}
	P(1)&=\Loe{1}, & P(2)&=\Loe{2}, &
	P(j)&=\Loe{ j \\ \vdots \\ 3 \\ 2\ \ 1 } \quad (3\le j\le 2m),
\end{align*}
while the indecomposable injective modules of $R$ are the modules
\begin{align*}
	I(1)&=\Loe{ 2m \\ \vdots \\ 3 \\ 1 }, &
	I(2)&=\Loe{ 2m \\ \vdots \\ 3 \\ 2 }, &
	I(j)&=\Loe{ 2m \\ \vdots \\ j } \quad (3\le j\le 2m).
\end{align*}
In particular $I(2m)=S(2m)$, and $\rad^{\,j-2}P(j)=S(1)\oplus S(2)$ for $j\ge3$.
We equip $Q_0=\{1,\dots,2m\}$ with the partial order $\preceq$ determined by
\begin{equation}\label{eq:order}
	2m \;\succ\; 2m-1 \;\succ\;\cdots\;\succ\; m+2
	\;\succ\; \{1,2\} \;\succ\; 3  \;\succ\;\cdots\;\succ\; m+1,
\end{equation}
where $\{1,2\}$ abbreviates that both $1$ and $2$ satisfy the displayed
relations, while $1$ and $2$ are incomparable to each other. 
Since $R$ is a hereditary algebra, $(R, Q_0, \preceq)$ is a quasi-hereditary algebra. 
With respect to this ordering, the standard modules  are 	\[
\Delta(i)=\begin{cases} S(i), & 1\le i\le m+1,\\[2pt] P(i), & m+2\le i\le 2m,\end{cases}
\]
and $\Delta(1)=P(1)$, $\Delta(2)=P(2)$ since $P(1)=S(1)$ and $P(2)=S(2)$. 
	The costandard modules are
	\begin{align*}
		\nabla(1)=\Loe{ m+1 \\ \vdots \\ 3 \\ 1 },\qquad
		\nabla(2)=\Loe{ m+1 \\ \vdots \\ 3 \\ 2 }, \qquad
		\nabla(j)=\Loe{ m+1 \\ m \\ \vdots \\ j }\ \ (3\le j\le m), \qquad
	\nabla(i)=S(i)\ \ (m+1\le i\le 2m).
	\end{align*}

\begin{Theorem}\label{RingelselfdualityforD}
	The quasi-hereditary algebra $(R, Q_0, \preceq)$ is Ringel self-dual.
\end{Theorem}

The idea of the proof is to apply Theorem \ref{maintheoremoneimplication}. To do that we need to verify that all $\tau$-orbits of $R$ have exactly $2m-1$ indecomposable modules and thus that additive closure of the characteristic tilting module lies exactly at the center of the Auslander-Reiten quiver of $R$.

\begin{Lemma}Let $C$ be the Cartan matrix of $R$. \label{coxeterlemma}
	The Coxeter matrix $\Phi=-C^TC^{-1}$ of $R$ in terms of the standard basis $\varepsilon_1, \ldots, \varepsilon_{2m}$ is given by 
		\[
	\Phi \varepsilon_1=-\varepsilon_1-\sum_{i=3}^{2m}\varepsilon_i,\qquad
	\Phi \varepsilon_2=-\varepsilon_2-\sum_{i=3}^{2m}\varepsilon_i,\qquad
	\Phi \varepsilon_3=\sum_{i=1}^{2m}\varepsilon_i,\qquad
	\Phi \varepsilon_j=\varepsilon_{j-1}\ (4\le j\le 2m),
	\] while the inverse of the Coxeter matrix is given by 
	\[
	\Phi^{-1}\varepsilon_1=\varepsilon_2+\varepsilon_3,\qquad
	\Phi^{-1}\varepsilon_2=\varepsilon_1+\varepsilon_3,\qquad
	\Phi^{-1}\varepsilon_j=\varepsilon_{j+1}\ (3\le j\le 2m-1),\qquad
	\Phi^{-1}\varepsilon_{2m}=-\sum_{i=1}^{2m}\varepsilon_i .\]
\end{Lemma}
\begin{proof}
	By reading off the Loewy series of the $P(j)$, $(i, j)$-entry of the Cartan matrix $C$ is given by
	\[
	C_{ij}=\dim_k\Hom_R(P(i), P(j))=[P(j):S(i)]=
	\begin{cases}
		1, & i=j\in\{1,2\},\\
		1, & j\ge 3 \text{ and } i\le j,\\
		0, & \text{otherwise}
	\end{cases}.\] It is not difficult to see that $j$-th column of $C^{-1}$ is $\varepsilon_1$ for $j=1$, $\varepsilon_2$ for $j=2$,
$\varepsilon_3-\varepsilon_1-\varepsilon_2$ for $j=3$ and so by induction $\varepsilon_j-\varepsilon_{j-1}$ for $j\ge4$, that is, the inverse of the Cartan matrix is given by
	\[
(C^{-1})_{ij}=
\begin{cases}
	1, & i=j,\\
	-1, & (i,j)\in\{(1,3),(2,3)\},\\
	-1, & 3\le i\le 2m-1,\ j=i+1,\\
	0, & \text{otherwise.}
\end{cases}.
\]
So the claim about the Coxeter matrix follows by just computing $-C^TC^{-1}$. The statement about the inverse of the Coxeter matrix is
checked directly by applying $\Phi$ to the right-hand sides.
\end{proof}

Observe that for every non-negative number $i\in \{0, \ldots, 2m-2\}$ the following equalities hold:
\begin{align}
	\Phi^{-i} \varepsilon_1=(\Phi^{-1})^i \varepsilon_1=\sum_{k=3}^{i+2} \varepsilon_k +\begin{cases}
		\varepsilon_1, & i \text{ is even}, \\
		\varepsilon_2 & i \text{ is odd}
	\end{cases} \label{eq7}
	\\
	\Phi^{-i} \varepsilon_2=(\Phi^{-1})^i \varepsilon_2=\sum_{k=3}^{i+2} \varepsilon_k+\begin{cases}
		\varepsilon_2, & i \text{ is even}, \\
		\varepsilon_1 & i \text{ is odd}
	\end{cases} \label{eq8}
\end{align} 
In particular, for $i=2m-2$ we get that $\tau^{-i} P(1)=I(1)$ and $\tau^{-i} P(2)=I(2)$. So both the $\tau$-orbits of $P(1)$ and $P(2)$ have exactly $2m-1$ indecomposable $R$-modules. Moreover, this holds true for every $\tau$-orbit.

\begin{Lemma}\label{tauorbitsinspecialcaseD}
	Every $\tau$-orbit contains exactly $2m-1$ indecomposable $R$-modules.
\end{Lemma}
\begin{proof}
	 The automorphism group of the Dynkin diagram $D_{2m}$ is $\{\id, (12)\}$ when $m>2$, so by \eqref{eq7}, \eqref{eq8}	and Lemma \ref{sizetauorbitsADE} the result follows whenever $m>2$. 
	 Assume now that $m=2$. 
	 Denote by $l_i$ the number of indecomposable modules in the $\tau$-orbit of $P(i)$. By \eqref{eq7} and \eqref{eq8} we get that $l_1=l_2=3$. By Lemma \ref{sizetauorbitsADE}, $l_3=3$. Since there are 12 indecomposable modules over $kD_4$ it follows that $l_4=3$ and so the statement also holds true for $m=2$.
\end{proof}

\begin{Lemma}\label{characteristictiltingmoduleofD}
	The characteristic tilting module $T$ of $(R, Q_0, \preceq)$  is isomorphic to $\tau^{-m+1}(R)=\tau^{m-1}(DR)$. 
\end{Lemma}
\begin{proof}
Given $M\in R\m$ we write $m(M):=\min\{m\geq 0\colon \tau^{m}M\ \text{is projective}\}$. By \eqref{eq7} and \eqref{eq8}, it follows that $\tau^{m-1}\Cs(i)$ is either $S(1)$ or $S(2)$ for $i=1, 2$. So $\tau^{m-1}\Cs(i)$ is projective for $i=1, 2$, that is, $m(\Cs(i))=m-1$ for $i=1, 2$. The simple modules $S(3), \ldots, S(2m)$ are all in the same $\tau$-orbit. Indeed, $\Phi \varepsilon_j=\varepsilon_{j-1}$ for $j\in \{4, \ldots, 2m\}$, that is $\tau S(j)=S(j-1)$. 
Thus $m(S(j))=1+m(S(j-1))$ for $j=4, \ldots, 2m$. By Lemma \ref{coxeterlemma}, $\tau S(3)=P(2m)$. By induction, it follows that
\begin{align}
	m(S(j))=j-2, \qquad 3\leq j\leq 2m. \label{eq9}
\end{align}
It follows that $m(\St(i))\leq m-1$ for every $i\in Q_0$. Thus, if $\Ext^1(\St(i), \tau^{-(m-1)}(R))\neq 0$ for some $i\in Q_0$, then by Lemma \ref{taumeasure}, we would get $m-1=m(\tau^{-(m-1)}(R))<m(\St(i))\leq m-1$. Hence, $\tau^{-(m-1)}(R)\in \mathcal{F}(\Cs)$. 

Observe that $\Cs(1)\twoheadrightarrow \Cs(i)$ for all $i=3, \ldots, m$. By Lemma \ref{taumeasure}, $m-1=m(\Cs(1))\leq m(\Cs(i))$ for all $i=3, \ldots, m$. By \eqref{eq9}, 
$m(\Cs(i))=m(S(i))=i-2\geq m-1$ for $m+1\leq i\leq 2m.$ So, $m(\Cs(i))\geq m-1$ for all $i\in Q_0$. By Lemma \ref{taumeasure}, it follows that $\tau^{-(m-1)}(R)\in {}^{\perp} \mathcal{F}(\Cs)$ and thus $\tau^{-(m-1)}(R)\in \mathcal{F}(\St)\cap \mathcal{F}(\Cs)$. Since $\tau^{-(m-1)}(R)$ contains exactly $2m$ non-isomorphic indecomposable modules $\tau^{-(m-1)}(R)$ is isomorphic to the characteristic tilting module. By Lemma \ref{tauorbitsinspecialcaseD}, $\tau^{-(m-1)}(\tau^{-(m-1)}(R))=\tau^{-(2m-2)}(R)=DR$ and thus the result follows.
\end{proof}

\begin{proof}[Proof of Theorem \ref{RingelselfdualityforD}]
	The result follows from Lemma \ref{characteristictiltingmoduleofD} and Theorem \ref{maintheoremoneimplication}. 
\end{proof}

\begin{Example}
	For $m=2$, $Q$ corresponds to the quiver
	\begin{equation}\label{eq:D2m-quiver}
		\begin{tikzcd}[row sep=small, column sep=small]
			 & & 1 \\
			4 \ar[r] & 3 \ar[ru]\ar[rd] & \\
			 & &  2
		\end{tikzcd}
	\end{equation} and $R$ is the path algebra $kQ$, where $Q_0$ is equipped with the order $\preceq$ determined by $4 \;\succ\; \{1, 2\} \;\succ\; 3$.
The Auslander-Reiten quiver of $R$ is the following

\noindent\makebox[\linewidth][l]{\resizebox{\textwidth}{!}
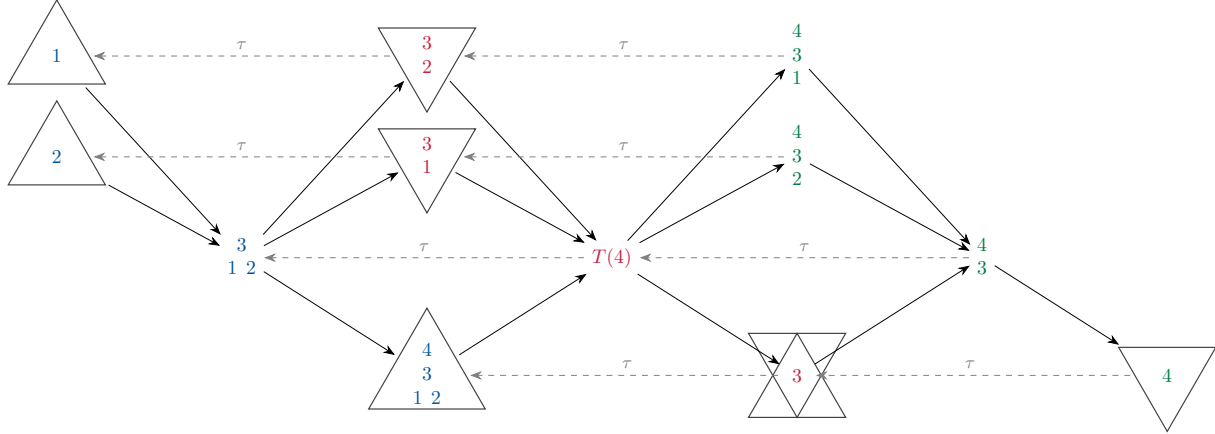
{%
	\begin{tikzpicture}[x=3.3cm,y=1.5cm]
	
	\node[mod,projcol] (nP1) at (-0.2, 2.4) {$1$};
	\node[mod,projcol] (nP2) at (-0.2, 1.2) {$2$};
	\node[mod,projcol] (nP3) at (0.8, 0)   {$\Loe{3\\1\;\,2}$};
	\node[mod,projcol] (nP4) at (1.8,-1.4) {$\Loe{4\\3\\1\;\,2}$};
	
	\node[mod,tiltcol] (nT2) at (1.8, 2.4) {$\Loe{3\\2}$};
	\node[mod,tiltcol] (nT1) at (1.8, 1.2) {$\Loe{3\\1}$};
	\node[mod,tiltcol] (nT4) at (2.8, 0)   {$T(4)$};
	\node[mod,tiltcol] (nT3) at (3.8,-1.4) {$3$};
	
	\node[mod,injcol] (nI1) at (3.8, 2.4) {$\Loe{4\\3\\1}$};
	\node[mod,injcol] (nI2) at (3.8, 1.2) {$\Loe{4\\3\\2}$};
	\node[mod,injcol] (nI3) at (4.8, 0)   {$\Loe{4\\3}$};
	\node[mod,injcol] (nI4) at (5.8,-1.4) {$4$};

	\node[stdmark,minimum size=2cm] (mP1) at (nP1) {};   
	\node[stdmark,minimum size=2cm] (mP2) at (nP2) {};   
	\node[stdmark,minimum size=2.4cm] (mP4) at (nP4) {};   
	\node[stdmark,minimum size=2cm,yshift=-2.5mm] (mT3) at (nT3) {};
	
	\node[costdmark,minimum size=2cm] (mT2)  at (nT2) {};
	\node[costdmark,minimum size=2cm] (mT1)  at (nT1) {};
	\node[costdmark,minimum size=2cm] (mI4)  at (nI4) {};
	\node[costdmark,minimum size=2cm,yshift=2.5mm] (mT3b) at (nT3) {};
	
	\draw[ar] (mP1) -- (nP3);
	\draw[ar] (mP2) -- (nP3);
	\draw[ar] (nP3) -- (mP4);
	\draw[ar] (nP3) -- (mT2);
	\draw[ar] (nP3) -- (mT1);
	\draw[ar] (mP4) -- (nT4);
	\draw[ar] (mT2) -- (nT4);
	\draw[ar] (mT1) -- (nT4);
	\draw[ar,shorten >=5pt] (nT4) -- (nT3);
	\draw[ar] (nT4) -- (nI1);
	\draw[ar] (nT4) -- (nI2);
	\draw[ar,shorten <=5pt] (nT3) -- (nI3);
	\draw[ar] (nI1) -- (nI3);
	\draw[ar] (nI2) -- (nI3);
	\draw[ar] (nI3) -- (mI4);
	
	\draw[taua] (mT2) -- node[above,gray,font=\small] {$\tau$} (mP1);
	\draw[taua] (nI1) -- node[above,gray,font=\small] {$\tau$} (mT2);
	\draw[taua] (mT1) -- node[above,gray,font=\small] {$\tau$} (mP2);
	\draw[taua] (nI2) -- node[above,gray,font=\small] {$\tau$} (mT1);
	\draw[taua] (nT4) -- node[above,gray,font=\small] {$\tau$} (nP3);
	\draw[taua] (nI3) -- node[above,gray,font=\small] {$\tau$} (nT4);
	\draw[taua,shorten <=5pt] (nT3) -- node[above,gray,font=\small] {$\tau$} (mP4);
	\draw[taua,shorten >=5pt] (mI4) -- node[above,gray,font=\small] {$\tau$} (nT3);
\end{tikzpicture}}}

Here, the standard (resp. costandard) modules are encircled by $\Delta$ (resp. $\Cs$), at blue the modules in $\cogen \tau T$, at green the modules in $\gen \tau^{-1}T$, and at red the modules in the additive closure of the characteristic tilting module $T$. The module $T(4)$ corresponds is the indecomposable module defined by the property $\dv T(4)=\varepsilon_1+\varepsilon_2+2\varepsilon_3+\varepsilon_4$.
The indecomposable modules of $\mathcal{F}(\St)$ (resp. $\mathcal{F}(\Cs)$) are the modules displayed as blue or red (resp. green or red). The action of $\tau$ on the Auslander-Reiten quiver is displayed with dashed arrows.
By Theorem \ref{maintheoremoneimplication}, it is clear that the algebra is Ringel self-dual and the exact equivalence $\Psi\colon \mathcal{F}(\St)\rightarrow \mathcal{F}(\Cs)$ satisfies
\[\Psi(\St(1))=\Cs(2), \quad \Psi(\St(2))=\Cs(1), \quad \Psi(\St(3))=\Cs(4), \quad \Psi(\St(4))=\Cs(3).  \qedhere  \]
\end{Example}

This case finishes the proof of Theorem \ref{maintheoremBproof}:

\begin{proof}[Proof of Theorem \ref{maintheoremBproof}]
	A simply laced Dynkin diagram is either a diagram of type $A$, $D$ or $E$. So the result follows by combining Theorem \ref{thm4dot20}, Theorem \ref{RingelselfdualityforD}, Corollary \ref{RsdfortypeE78}, Corollary \ref{cor4dot5}, Corollary \ref{corD4} and Corollary \ref{excludingtypeA}.
\end{proof}

The Ringel self-dual algebras with underlying quiver of type $D_4$ always have a characteristic tilting module like in type $A$, that is, it is isomorphic to a power of the inverse Auslander-Reiten translation of the regular module.

\begin{Prop}\label{tiltingcaseD4}
	Fix an orientation on $D_4$ and a partial order $\leq$ on $Q_0=\{1, 2, 3, 4\}$ so that $(A, Q_0, \leq)$ is Ringel self-dual, where $A:=kD_4$. Then every $\tau$-orbit contains exactly $3$ indecomposable modules and the characteristic tilting module is $\tau^{-1}(A)\cong \tau (DA)$.
\end{Prop}
\begin{proof}
	By Lemma \ref{sizetauorbitsADE}, every $\tau$-orbit contains exactly $3$ indecomposable modules. The automorphism group of $D_4$ is the symmetric group $S_3$. Let $\sigma$ be the automorphism graph given in Lemma \ref{QuiveraddT} and consider the same notation as in Lemma \ref{QuiveraddT}. Assume that $3$ is the central vertex. If $\sigma$ is a $3$-cycle, then  Lemma \ref{QuiveraddT} yields $3$ equations $2=p_i+p_{\sigma^{-1}(i)}$ for $i=1, 2, 4$ and $p_3=1$. It follows that $p_1=p_4=p_2=1$. So, the characteristic tilting module is of the desired form. If $\sigma$ is the identity, then also $p_i=1$ for all $i\in \{1, 2, 3, 4\}$. So it remains to check what happens for $\sigma\in \{(12), (14), (24)\}$. Up to relabelling of the vertices, we can assume without loss of generality that $\sigma=(12)$. Then $p_4=p_3=1$ and $p_1+p_2=2$. Assume that $p_1=0$ and $p_2=2$. So $T_1$ is projective and $T_2$ is injective. Hence, $T_1=P(1)$ and $T_2=I(2)$. Also $p_3=p_1+1$ and $p_3=p_2-1$, thus by Lemma \ref{QuiveraddT} the orientation must be $1\rightarrow 3\rightarrow 2$. Without loss of generality we can assume that $Q$ corresponds to 
	\begin{equation}
		\begin{tikzcd}[row sep=small, column sep=small]
			& & 1  \ar[dl]\\
			4 \ar[r] & 3\ar[rd] & \\
			& &  2
		\end{tikzcd},
	\end{equation} otherwise we could just consider the opposite algebra and relabel the vertices. Thus the candidate for characteristic tilting module is $P(1)\oplus I(2)\oplus \tau^{-1} P(3)\oplus \tau^{-1}P(4)$. But none of these direct summands is a simple module since $\dim \tau^{-1}P(4)=2$, $\dim \tau^{-1}P(3)=5$, $\dim I(2)=3=\dim P(1)$. So such a case cannot occur and so for transpositions we also have $p_1=p_2=p_3=p_4=1$.
\end{proof}

\begin{Cor}\label{corCinternal}
	Let $(A, \Lambda, \leq)$ be a (connected and basic) quasi-hereditary $k$-algebra with $\gldim A\leq 1$ whose Gabriel quiver is not of type $D_{2m}$ for $m>2$. Then
	$(A, \Lambda, \leq)$ is Ringel self-dual and it is an algebra of finite representation type if and only if the (basic) characteristic tilting module $T$ is isomorphic to $\tau^{-n}(A)=\tau^n(DA)$ for some $n\in \mathbb{N}\cup \{0\}$.
\end{Cor}
\begin{proof}
	The result follows from Propositions \ref{tiltingcaseD4}, \ref{Proptauorbitsprediction} and Theorems \ref{maintheoremoneimplication} and \ref{maintheoremBproof}.
\end{proof}

In type $D$, Lemma \ref{QuiveraddT} is not enough to enforce that the characteristic tilting module is isomorphic to a power of the inverse Auslander-Reiten translation of the regular module, that is, the analogue of Proposition \ref{Proptauorbitsprediction} for type $D_{2m}$ with $m>2$ involves a twist induced by the non-trivial graph automorphism of $D_{2m}$.

\begin{Prop} Let $Q$ be a quiver with underlying Dynkin diagram $D_{2m}$ with $m>2$, let $A$ be the path algebra $kQ$ and $\leq $ a partial order on $Q_0$ so that $(A, Q_0, \leq)$ is Ringel self-dual. If $T$ is the (basic) characteristic tilting module, then $T\cong \displaystyle X\oplus \bigoplus_{i=3}^{2m} \tau^{-m+1}P(i)\cong Y\oplus \bigoplus_{i=3}^{2m} \tau^{m-1}I(i)$,  where $X\cong \tau^{-p} P(1)\oplus \tau^{-(2m-2-p)}P(2)$ and $Y\cong \tau^{2m-2-p}I(1)\oplus \tau^{p}I(2)$ with $p\in \{m-2, m-1, m\}$.\label{proptwistedcaseD}
\end{Prop}
\begin{proof} 
	Since $A$ is hereditary algebra of finite representation type, we can write the characteristic tilting module $T$ of $A$ is isomorphic to $\oplus_{i=1}^{2m} \tau^{-p_i} P(i)$ for some $p_i\in \mathbb{N}\cup \{0\}$. By (b)(ii) of Lemma \ref{QuiveraddT} and Lemma \ref{sizetauorbitsADE},
	\begin{equation}
		p_i+p_{\sigma(i)}+1=2m-1, \quad \forall i=1, \ldots, 2m, \label{eq15}
	\end{equation} $\sigma$ is a graph automorphism of $D_{2m}$. Assuming that $D_{2m}$ is labelled like in (8), the possible graph automorphisms are then the identity or the permutation $(12)$. Thus, by \eqref{eq15} we get that $p_i=m-1$ for every $i=3, \ldots, 2m$. If $\sigma=\id$, then Lemma \ref{QuiveraddT} (b)(ii) gives that $p_1=p_3=p_2=m-1$. Assume now that $\sigma=(12)$. If $1\rightarrow 3\leftarrow 2$ or $1\leftarrow 3\rightarrow 2$ is a subquiver of $Q$, then $(12)$ preserves the orientation of these arrows and Lemma \ref{QuiveraddT} b(ii) gives again that $p_1=p_3=p_2=m-1$. On the other hand, if $1\rightarrow 3\rightarrow2$ then $(12)$ reverses the orientation of these two arrows and so we get $m-1=p_3=p_1+1$ and $p_2=m$ by Lemma \ref{QuiveraddT} (b)(iii). In contrast, if $1\leftarrow 3\leftarrow 2$, then Lemma \ref{QuiveraddT} (b)(iii) gives $p_1=p_3+1=m$ and $p_2=m-2$.
So, the (basic) characteristic tilting module is isomorphic to $X\oplus \oplus_{i=3}^{2m} \tau^{-m+1}P(i)$, where $X\cong \tau^{-p_1}P(1)\oplus \tau^{-p_2}P(2)$ with $p_1+p_2=2m-2$ and $p_1\in \{m-2, m-1, m\}$.
 The second claimed isomorphism follows then by Equation \eqref{eq15}  and Lemma \ref{QuiveraddT}.  \end{proof}

\subsection{Open questions}

We leave it as an open problem to determine the full list of orientations of $Q\in \{D_{2m}, A_{4m+1}\}$ for which  there exists a partial ordering $\leq$ on $Q_0$ such that $(kQ, Q_0, \leq)$ is Ringel self-dual. Moreover, we can pose the following questions:

\begin{itemize}
	\item How many different orientations of $A_{4m+1}$ yield a path algebra with a Ringel self-dual quasi-hereditary structure?
	\item How many different orientations of $D_{2m}$ yield a path algebra with a Ringel self-dual quasi-hereditary structure?
	\item Is there a version of Theorem \ref{maintheoremoneimplication} that covers the twisted cases of type $D$?
	\item What happens to hereditary algebras of infinite representation type under Ringel self-duality? 
\end{itemize}

\section*{Acknowledgments}
The author thanks Steffen Koenig for his comments on an earlier version of  this manuscript.

\bibliographystyle{alpha}
\bibliography{bibarticle}

\Address
\end{document}